\documentclass[11pt]{article}

\usepackage[margin=1.1in]{geometry}
\usepackage{amsmath,amssymb,amsthm}
\usepackage{mathtools}
\usepackage{microtype}
\usepackage[colorlinks=true,linkcolor=blue,citecolor=blue,urlcolor=blue]{hyperref}

\numberwithin{equation}{section}
\newtheorem{theorem}{Theorem}[section]
\newtheorem{lemma}[theorem]{Lemma}

\newtheorem{proposition}[theorem]{Proposition}
\theoremstyle{remark}
\newtheorem{remark}[theorem]{Remark}
\newtheorem{remarks}[theorem]{Remarks}

\allowdisplaybreaks

\title{Sharp Lifespan Estimates for a Semilinear Wave Equation\\
with Nonlinear Damping}
\author{Firas Kaabi\\
\small Department of Mathematics, Faculty of Sciences of Tunis,\\[-2pt]
\small University of Tunis El Manar, LR Analyse Non-Lin\'eaire et G\'eom\'etrie, LR21ES08,\\[-2pt]
\small El Manar 2, 2092 Tunis, Tunisia\\[-2pt]
\small \texttt{firaskaabi17@gmail.com}}
\date{}

\begin{document}
\maketitle

\begin{abstract}
We study the maximal existence time $T^{*}(\varrho)$ of the solution of the
semilinear wave equation $u_{tt}-\Delta u=u|u|^{p-1}-u_{t}|u_{t}|^{q-1}$ in a
bounded domain, with Dirichlet boundary condition and initial data
$(\varrho f,\varrho g)$, where $1<q<p$ and the amplitude $\varrho$ is large.
For nontrivial $f$ and sufficiently large $\varrho$, the concavity method gives
$T^{*}(\varrho)\leq C\varrho^{-(p-1)/2}$ for $q\leq2p/(p+1)$ and
$T^{*}(\varrho)\leq C\varrho^{-(p-q)/q}$ for $q>2p/(p+1)$, whereas the energy
method gives a lower bound of order $\varrho^{1-p}$ only.  We prove lower
bounds with the same exponents as the upper ones, so that
$T^{*}(\varrho)\asymp\varrho^{-\vartheta(p,q)}$ with
$\vartheta(p,q)=\min\{(p-1)/2,\,(p-q)/q\}$.  The proof rests on a hyperbolic
rescaling which converts the large amplitude into a dilation of the domain and
a coefficient $\varrho^{(q(p+1)-2p)/2}$ in front of the damping term, on a
local existence theory in uniformly local energy norms whose existence time
does not depend on that coefficient, and on a quantitative use of the
dissipation when the coefficient is large.  The threshold $2p/(p+1)$ is the
value of $q$ at which the damping term is invariant under the rescaling.  No
attempt is made to optimize the constants: sharpness is meant throughout at
the level of the exponent.
\end{abstract}



\section{Introduction}\label{sec:intro}

Let $\Omega$ be a bounded Lipschitz domain of $\mathbb{R}^{n}$, with $n\geq1$.
We consider the semilinear wave equation
\begin{equation}\label{eq:main}
\begin{cases}
u_{tt}-\Delta u=u|u|^{p-1}-u_{t}|u_{t}|^{q-1},
& (x,t)\in\Omega\times(0,T^{*}),\\
u(x,0)=u_{0}(x)=\varrho f(x),\quad u_{t}(x,0)=u_{1}(x)=\varrho g(x),
& x\in\Omega,\\
u(x,t)=0, & (x,t)\in\partial\Omega\times(0,T^{*}),
\end{cases}
\end{equation}
where $T^{*}=T^{*}(\varrho)$ is the maximal existence time of the solution, in
the class of weak solutions recalled in Section \ref{sec:results}, where
$\varrho>0$ is the amplitude of the initial data, and where the profiles $f$
and $g$ are fixed and satisfy
\begin{equation}\label{eq:data}
f\in H_{0}^{1}(\Omega)\cap L^{\infty}(\Omega),
\qquad
\nabla f\in L^{n\vee2}(\Omega),
\qquad
g\in L^{n\vee2}(\Omega),
\end{equation}
with $n\vee2=\max\{n,2\}$.  The exponents $p$ and $q$ satisfy
\begin{equation}\label{eq:exponents}
\begin{cases}
1<q<p<\infty, & \text{if }n=1\text{ or }n=2,\\
1<q<p<\dfrac{n}{n-2}, & \text{if }n\geq3.
\end{cases}
\end{equation}
Since $\Omega$ is bounded, \eqref{eq:data} implies that
$(f,g)\in H_{0}^{1}(\Omega)\times L^{2}(\Omega)$, so that the data belong to
the energy space; the role of the additional integrability required in
\eqref{eq:data} is explained in Remarks \ref{rem:main} below.

The study of \eqref{eq:main} is motivated by the modelling of structures
subject to large dynamic loads, such as bridges, towers or airplane wings, in
which nonlinear dissipative devices are used to control the vibrations.
Magnetorheological dampers, for instance, produce a force which depends on the
velocity in a nonlinear manner \cite{Spencer1997,Ikhouane2007}, and the
identification and the design of such devices is an active subject
\cite{Kerschen2006,Lazar2013}.  In \eqref{eq:main} a nonlinear source term
competes with a velocity-dependent nonlinear damping, and the amplitude of the
initial disturbance is a natural parameter; the question of how the time of
existence of the solution depends on it is the object of this paper.

The blow-up phenomenon for the semilinear wave equation has been studied by
many authors, both in the whole space and in bounded domains.  For the
undamped equation $u_{tt}-\Delta u=u|u|^{p-1}$, detailed results on the
blow-up rate and on the blow-up profile have been obtained in several settings,
notably in one space dimension, by Merle and Zaag
\cite{MerleZaag2003,MerleZaag2005,MerleZaag2007,MerleZaag2012a,MerleZaag2016}
and by C\^ote and Zaag \cite{CoteZaag2013}.  Ball \cite{Ball1977} showed
that, without damping, the source term forces the solutions with negative
initial energy to blow up in finite time, and Haraux and Zuazua
\cite{Haraux1988} showed that, without source, the damping term produces
global solutions for arbitrary data.  When both terms are present, Levine
\cite{Levine1974} obtained blow-up for negative initial energy in the linearly
damped case, and Georgiev and Todorova \cite{Georgiev1994} developed the local
and global existence theory in the energy class and proved global existence
when $q\geq p$.  Messaoudi
\cite{Messaoudi2001} proved that the solutions with negative initial energy
blow up in finite time whenever $q<p$, and Vitillaro \cite{Vitillaro2000}
extended this result to a range of positive initial energies bounded by an
explicit constant.  General nonexistence theorems for abstract evolution
equations with dissipation were obtained by Levine and Serrin
\cite{LevineSerrin1997}, and refined well-posedness results in the simultaneous
presence of source and damping can be found in
\cite{TodorovaVitillaro2005,BociuLasiecka2010}.  Equations of the form
\eqref{eq:main} with variable exponents were treated in
\cite{Messaoudi2017,SunRenGao2016,Kafini2019}.

Once blow-up is known to occur, the natural question is to estimate the
blow-up time itself, from above and from below.  For hyperbolic problems the
derivation of lower bounds is more delicate than for parabolic ones, since the
comparison techniques available in the parabolic setting do not apply.
Lower bounds for the blow-up time of damped and related wave equations have
been obtained in \cite{SunGuoGao2014,Zhou2015,Liu2023}, and we refer to
\cite{Zhou2007,Takamura2015} for lifespan estimates for the undamped equation.
For initial data of the form $(\varrho f,\varrho g)$, the quadratic
phase-space argument recalled below gives a lower bound of order
$\varrho^{-(p-2)}$ for a strongly damped equation \cite{Kaabi2026a}, and a two-sided estimate for a
semilinear wave equation with fractional structural damping is obtained in
\cite{Kaabi2026b} by a mechanism in which the dissipation cancels in the
concavity functional.
In \cite{BHK}, Bchatnia, Hamouda and Kaabi studied the same equation
\eqref{eq:main} for large values of the amplitude $\varrho$.  Under
conditions of Vitillaro type on the initial data, recalled in
\eqref{eq:hyp-blowup} below, which are satisfied for every sufficiently large
$\varrho$ as soon as $f\not\equiv0$, the solution blows up in finite time.  The
upper bound for $T^{*}(\varrho)$ was obtained there by the concavity method:
writing $E$ for the energy and $E_{1}$ for the depth of the potential well, the
functional
\[
\gamma(t)=\bigl(E_{1}-E(t)\bigr)^{1-\beta}+\theta\int_{\Omega}u(t)u_{t}(t)\,dx
\]
satisfies a differential inequality $\gamma'\geq\omega\gamma^{1/(1-\beta)}$,
whose integration bounds $T^{*}$ by a negative power of $\gamma(0)$, and
$\gamma(0)$ is of order $\varrho^{(p+1)(1-\beta)}$.  The exponent $\beta$
produced by the argument is $\frac{p-1}{2(p+1)}$ when $q\leq q_{c}$ and
$\frac{p-q}{q(p+1)}$ when $q>q_{c}$, where
\begin{equation}\label{eq:qc}
q_{c}:=\frac{2p}{p+1},
\end{equation}
and this gives the upper bounds
\begin{equation}\label{eq:upper}
T^{*}(\varrho)\leq C_{u}\varrho^{-(p-1)/2}\quad\text{if }1<q\leq q_{c},
\qquad
T^{*}(\varrho)\leq C_{u}'\varrho^{-(p-q)/q}\quad\text{if }q_{c}<q<p.
\end{equation}
Since $p>1$, the inequalities $2p>p+1$ and $2p<p(p+1)$ give $1<q_{c}<p$, so
that both ranges of $q$ occurring in \eqref{eq:upper} are non-empty.  These are
the upper estimates obtained in \cite{BHK}.  Appendix
\ref{sec:appendix} goes through that proof once with the dependence of the
constants on $\varrho$ made explicit, which is what the two-sided estimate of
Theorem \ref{thm:two-sided} requires; see Remark \ref{rem:BHKstatement}.

The lower bound in \cite{BHK} was obtained by an energy argument on the fixed
domain $\Omega$.  The quantity
$\mathcal{M}(t)=\|\nabla u(t)\|_{2}^{2}+\|u_{t}(t)\|_{2}^{2}$ satisfies
$\mathcal{M}'\leq C\mathcal{M}^{(p+1)/2}$, and therefore
$T^{*}\geq C\mathcal{M}(0)^{-(p-1)/2}$; since $\mathcal{M}(0)$ is of order
$\varrho^{2}$, this yields $T^{*}(\varrho)\geq C\varrho^{1-p}$ in both regimes.
As $(p-1)/2<p-1$, the two estimates do not match, and the gap between the
exponents grows with $p$.

The energy argument stops at $\varrho^{1-p}$ for the following reason.  The
only information about the data which enters it is their size
in the energy norm: for $f\not\equiv0$ one has
$\mathcal{M}(0)\asymp\varrho^{2}$, and the differential inequality
$\mathcal{M}'\leq C\mathcal{M}^{(p+1)/2}$ then yields only the time scale
$\mathcal{M}(0)^{-(p-1)/2}\asymp\varrho^{1-p}$, whatever the profiles.  What is
lost is the shape of the profile $f$, in particular the fact that its
amplitude, and not only its energy, is of order $\varrho$.  The purpose of the
present paper is to recover this information and to prove lower bounds for
$T^{*}(\varrho)$ with the same exponents as in \eqref{eq:upper}, so that the
lifespan is determined exactly, up to multiplicative constants, in each of the
two regimes.

Our approach consists in rescaling the problem before applying any existence
theory.  Setting
\begin{equation}\label{eq:scaling}
\lambda=\varrho^{(p-1)/2},
\qquad
v(y,s)=\varrho^{-1}u\Bigl(\frac{y}{\lambda},\frac{s}{\lambda}\Bigr),
\end{equation}
problem \eqref{eq:main} becomes
\begin{equation}\label{eq:rescaled}
v_{ss}-\Delta v=v|v|^{p-1}-\varepsilon_{\varrho}\,v_{s}|v_{s}|^{q-1},
\qquad
\varepsilon_{\varrho}=\varrho^{\kappa},
\qquad
\kappa=\frac{q(p+1)-2p}{2},
\end{equation}
on the dilated domain $\Omega_{\lambda}=\lambda\Omega$.  The amplitude of the
rescaled data is of order one; the price to pay is that the domain expands and
that the damping coefficient is no longer equal to one.  The exponent $q_{c}$ of
\eqref{eq:qc} is precisely the value of $q$ at which the damping term is invariant
under \eqref{eq:scaling}, that is, at which $\kappa=0$: for $q<q_{c}$ one has
$\varepsilon_{\varrho}\to0$, and for $q>q_{c}$ one has
$\varepsilon_{\varrho}\to\infty$ as $\varrho\to\infty$.  This is the same
threshold as in \eqref{eq:upper}.

Two mechanisms are then available, and they produce the two lower bounds of
this paper.  The first one uses the damping term only through its sign.  Since
the map $\zeta\mapsto\zeta|\zeta|^{q-1}$ is monotone, the damping term enters
every local energy estimate with a favourable sign, so that the local
existence time of \eqref{eq:rescaled} depends on the data only through their
uniformly local energy norm, in the sense of Kato \cite{Kato1975}, and not on
the size of $\varepsilon_{\varrho}$.  The rescaled data have uniformly local
energy of order one, and finite propagation allows the local pieces to be
glued together.  This gives an existence time of order one in the rescaled
variables, hence
$T^{*}(\varrho)\geq C_{l}\varrho^{-(p-1)/2}$ for every $1<q<p$, which is
optimal when $q\leq q_{c}$ by \eqref{eq:upper}.

The second mechanism uses the size of the damping term as well, and concerns
the range $q>q_{c}$, in which the damping dominates the balance between the
three terms of the equation and in which the first lower bound is no longer
optimal.  We call $q>q_{c}$ the \emph{overdamped range}, the word referring to the size
of the exponent $q$ in \eqref{eq:main}.
Writing $A=\varepsilon_{\varrho}\to\infty$ and testing \eqref{eq:rescaled} with
$\chi^{2}v_{s}$, for a cutoff function $\chi$ of unit scale, the source term is
estimated against the dissipation by Young's inequality in the form
\begin{equation}\label{eq:young-intro}
\Bigl|\int\chi^{2}v|v|^{p-1}v_{s}\Bigr|
\leq\frac{1}{4}A\int\chi^{2}|v_{s}|^{q+1}
+CA^{-1/q}\int\chi^{2}|v|^{p(q+1)/q},
\end{equation}
and the commutator generated by the cutoff obeys an estimate of the same type.
The uniformly local energy of the rescaled solution therefore increases at the
rate $A^{-1/q}$, so that a rescaled time of order $A^{1/q}$ is needed before it
can double.  Combined with the local theory just described, this gives
existence up to the rescaled time $cA^{1/q}$, and since
\begin{equation}\label{eq:exponent-id}
\lambda^{-1}A^{1/q}=\varrho^{-(p-1)/2+\kappa/q}=\varrho^{-(p-q)/q},
\end{equation}
this is exactly the scale of the upper bound in \eqref{eq:upper}.  The exponent
$p(q+1)/q$ produced by \eqref{eq:young-intro} is admissible for the Sobolev
embedding in the whole range \eqref{eq:exponents}, and the argument requires
nothing of the data beyond \eqref{eq:data}.

The paper is organized as follows.  Section \ref{sec:results} fixes the notion
of solution and states the main results, Theorems \ref{thm:lower} and
\ref{thm:two-sided}.  Section \ref{sec:aux} contains the auxiliary results on
which the proofs rest: the rescaling and the size of the rescaled data in
uniformly local norms, the localized energy estimates, uniqueness and finite
propagation, the local existence theory in uniformly local norms, and the
growth estimate for the local energy when the damping coefficient is large.
Section \ref{sec:proof} contains the proofs of the main results.  Section
\ref{sec:open} collects some concluding remarks and open problems.  Appendix
\ref{sec:appendix} gives a complete derivation of the upper bounds
\eqref{eq:upper}, following the argument of \cite{BHK}, with constants whose
independence of $\varrho$ is verified explicitly, as is required for the
two-sided estimate.

\section{Main results}\label{sec:results}

This section is devoted to the statement of our results, preceded by the
notion of solution which is used throughout.  In the whole paper, $D$ denotes
a bounded Lipschitz domain of $\mathbb{R}^{n}$, $\varepsilon>0$ denotes a
damping coefficient, and $B_{r}(z)$ is the open ball of radius $r$ centred at
$z$.  We write $\|\cdot\|_{r}$ for the norm of $L^{r}$, and the time derivative
is denoted by $u_{t}$ or $u_{s}$ according to the name of the time variable in
use.  Unsubscripted constants $C$ are positive and may change from one
occurrence to the next; the numbered constants $C_{1},\dots,C_{8}$, together
with $C_{5}'$, are fixed once they are introduced.  No constant depends on
$\varrho$, on $\lambda$ or on $\varepsilon$ unless this is stated explicitly.

We work in the energy class of Georgiev and Todorova \cite{Georgiev1994}, see
also \cite{TodorovaVitillaro2005,BociuLasiecka2010}, namely
\begin{equation}\label{eq:class}
u\in C\bigl([0,T];H_{0}^{1}(D)\bigr)\cap C^{1}\bigl([0,T];L^{2}(D)\bigr),
\qquad
u_{t}\in L^{q+1}\bigl(D\times(0,T)\bigr).
\end{equation}
Since both the original problem and its rescaled version will be considered,
we state the weak formulation for the equation with a general damping
coefficient.  A function $u$ in the class \eqref{eq:class} is called a weak
solution of
\begin{equation}\label{eq:eqD}
u_{ss}-\Delta u=u|u|^{p-1}-\varepsilon\,u_{s}|u_{s}|^{q-1}
\quad\text{in }D\times(0,T),
\qquad u|_{\partial D}=0,
\end{equation}
if for all $0\leq s_{1}\leq s_{2}\leq T$ and every test function
$\varphi\in C([s_{1},s_{2}];H_{0}^{1}(D))\cap C^{1}([s_{1},s_{2}];L^{2}(D))$,
\begin{equation}\label{eq:weak}
\Bigl[\int_{D}u_{s}\varphi\,dy\Bigr]_{s_{1}}^{s_{2}}
+\int_{s_{1}}^{s_{2}}\!\!\int_{D}
\Bigl[-u_{s}\varphi_{s}+\nabla u\cdot\nabla\varphi
+\varepsilon\,u_{s}|u_{s}|^{q-1}\varphi-u|u|^{p-1}\varphi\Bigr]dy\,ds=0.
\end{equation}
Every term in \eqref{eq:weak} is finite.  Indeed
$u_{s}|u_{s}|^{q-1}\in L^{(q+1)/q}$, while
$\varphi\in C([s_{1},s_{2}];H_{0}^{1}(D))\hookrightarrow L^{q+1}$: if $n\leq2$
then $H_{0}^{1}(D)\hookrightarrow L^{r}(D)$ for every finite $r$, and if
$n\geq3$ then
\begin{equation}\label{eq:sobolev-chain}
q+1<p+1<\frac{n}{n-2}+1=\frac{2n-2}{n-2}<\frac{2n}{n-2}
\end{equation}
by \eqref{eq:exponents}.  Moreover $u|u|^{p-1}\in C([s_{1},s_{2}];L^{2}(D))$,
because $H_{0}^{1}(D)\hookrightarrow L^{2p}(D)$ under \eqref{eq:exponents}.
Nothing further is required of the notion of solution.  The two localized
inequalities used throughout, a localized energy identity and a localized
inequality for the difference of two solutions, are proved in
Lemmas \ref{lem:loc-energy} and \ref{lem:loc-diff} for every weak solution in
the class \eqref{eq:class}, between arbitrary times, hence on every compact
subinterval of the interval of existence.  Weak solutions in the class
\eqref{eq:class} are unique by Lemma \ref{lem:unique}, so that the maximal
existence time $T^{*}(\varrho)$ of the solution of \eqref{eq:main} is well
defined.

We shall use throughout the quantities introduced in \eqref{eq:scaling},
\eqref{eq:rescaled} and \eqref{eq:qc}, that is,
\begin{equation}\label{eq:params}
\lambda=\varrho^{(p-1)/2},
\qquad
\kappa=\frac{q(p+1)-2p}{2},
\qquad
\varepsilon_{\varrho}=\varrho^{\kappa},
\qquad
q_{c}=\frac{2p}{p+1},
\end{equation}
so that $\kappa<0$ if $q<q_{c}$, $\kappa=0$ if $q=q_{c}$, and $\kappa>0$ if
$q>q_{c}$.  In the last case we also write $A=\varepsilon_{\varrho}$.  The remaining standing symbols are introduced as follows: the dilated domain $\Omega_{\lambda}=\lambda\Omega$ and the
rescaled profiles $f_{\lambda},g_{\lambda}$ in Lemma \ref{lem:rescale}; the
maximal rescaled time $S^{*}_{\lambda}$ in \eqref{eq:times}; the cutoff $\chi$
in \eqref{eq:cutoff}; the uniformly local energy $\mathcal{E}$, its
localizations $\mathcal{E}_{z}$ and the dissipation $\mathcal{D}_{z}$ in
\eqref{eq:EandD}; the exponent $\Gamma$ in \eqref{eq:growth}; and, in Appendix
\ref{sec:appendix}, the Sobolev constant $B_{*}$ and the well depth $E_{1}$ in
\eqref{eq:E1}, the well function $\mathcal{W}$ in \eqref{eq:w-def} and the
concavity exponent $\beta$ in \eqref{eq:beta}.

Our first result gives the two lower bounds.  It holds in all regimes and
requires no condition ensuring blow-up, since it is a statement about
existence.

\begin{theorem}\label{thm:lower}
Assume \eqref{eq:data} and \eqref{eq:exponents}.
\begin{enumerate}
\item[$(i)$] There exists a constant $C_{l}>0$, depending only on $n$, $p$ and
the norms appearing in \eqref{eq:data}, and in particular independent of $q$,
such that
\begin{equation}\label{eq:lower1}
T^{*}(\varrho)\geq C_{l}\,\varrho^{-(p-1)/2}
\qquad\text{for all }\varrho\geq1.
\end{equation}
\item[$(ii)$] If moreover $q_{c}<q<p$, there exists a constant $C_{l}'>0$,
depending only on $n$, $p$, $q$ and the norms appearing in \eqref{eq:data},
such that
\begin{equation}\label{eq:lower2}
T^{*}(\varrho)\geq C_{l}'\,\varrho^{-(p-q)/q}
\qquad\text{for all }\varrho\geq1.
\end{equation}
\end{enumerate}
\end{theorem}

The two exponents coincide at $q=q_{c}$, and $(p-q)/q\leq(p-1)/2$ if and only
if $q\geq q_{c}$, since $2(p-q)\leq q(p-1)$ is equivalent to $2p\leq q(p+1)$;
thus \eqref{eq:lower2} improves \eqref{eq:lower1} exactly in the overdamped
range.  In order to combine Theorem \ref{thm:lower} with the upper bounds
\eqref{eq:upper}, we recall the hypotheses under which the latter were obtained
in \cite{BHK}.  The energy of the solution of \eqref{eq:main} is
\begin{equation}\label{eq:energy}
E(t)=\frac{1}{2}\|u_{t}(t)\|_{2}^{2}+\frac{1}{2}\|\nabla u(t)\|_{2}^{2}
-\frac{1}{p+1}\|u(t)\|_{p+1}^{p+1},
\qquad 0\leq t<T^{*}(\varrho),
\end{equation}
so that, with $u_{0}=\varrho f$ and $u_{1}=\varrho g$,
\begin{equation}\label{eq:E0}
E(0)=\frac{\varrho^{2}}{2}\bigl(\|\nabla f\|_{2}^{2}+\|g\|_{2}^{2}\bigr)
-\frac{\varrho^{p+1}}{p+1}\|f\|_{p+1}^{p+1}.
\end{equation}
Let $B_{*}$ be the optimal constant in the embedding
$\|u\|_{p+1}\leq B_{*}\|\nabla u\|_{2}$ for $u\in H_{0}^{1}(\Omega)$, and set
\begin{equation}\label{eq:E1}
\mu_{1}=\Bigl(\frac{1}{B_{*}^{p+1}}\Bigr)^{\frac{2}{p-1}},
\qquad
E_{1}=\frac{p-1}{2(p+1)}\,\mu_{1}.
\end{equation}
The hypotheses of \cite{BHK} are
\begin{equation}\label{eq:hyp-blowup}
\varrho^{2}\|\nabla f\|_{2}^{2}>\mu_{1}
\qquad\text{and}\qquad
E(0)<E_{1}.
\end{equation}
If $f\not\equiv0$, both conditions in \eqref{eq:hyp-blowup} hold for all
sufficiently large $\varrho$: the first one because
$\varrho^{2}\|\nabla f\|_{2}^{2}\to\infty$, and the second one because
$E(0)\leq C\varrho^{2}-c\varrho^{p+1}\to-\infty$, since $p+1>2$ and
$\|f\|_{p+1}>0$.  Under \eqref{eq:hyp-blowup} the solution of \eqref{eq:main}
blows up in finite time and satisfies \eqref{eq:upper}; this is proved in
Appendix \ref{sec:appendix}, following the argument of \cite{BHK}, in the form
of Proposition \ref{prop:upper}, which also provides constants independent of
$\varrho$.  Our second result is the two-sided estimate which follows.

\begin{theorem}\label{thm:two-sided}
Assume \eqref{eq:data}, \eqref{eq:exponents} and $f\not\equiv0$.  Then there
exists $\varrho_{0}\geq1$ such that, for every $\varrho\geq\varrho_{0}$, the
solution of \eqref{eq:main} blows up in finite time and
\begin{equation}\label{eq:two-sided-sub}
C_{l}\,\varrho^{-\frac{p-1}{2}}
\;\leq\;T^{*}(\varrho)\;\leq\;
C_{u}\,\varrho^{-\frac{p-1}{2}}
\qquad\text{if }1<q\leq\frac{2p}{p+1},
\end{equation}
and
\begin{equation}\label{eq:two-sided-over}
C_{l}'\,\varrho^{-\frac{p-q}{q}}
\;\leq\;T^{*}(\varrho)\;\leq\;
C_{u}'\,\varrho^{-\frac{p-q}{q}}
\qquad\text{if }\frac{2p}{p+1}<q<p.
\end{equation}
In other words, setting
\begin{equation}\label{eq:vartheta}
\vartheta(p,q)=\min\Bigl\{\frac{p-1}{2},\,\frac{p-q}{q}\Bigr\}
=\begin{cases}
\dfrac{p-1}{2}, & 1<q\leq\dfrac{2p}{p+1},\\[3mm]
\dfrac{p-q}{q}, & \dfrac{2p}{p+1}<q<p,
\end{cases}
\end{equation}
there exist constants $0<C\leq C'<\infty$ such that
\begin{equation}\label{eq:two-sided}
C\varrho^{-\vartheta(p,q)}\leq T^{*}(\varrho)\leq C'\varrho^{-\vartheta(p,q)},
\qquad\varrho\geq\varrho_{0}.
\end{equation}
\end{theorem}

The constant $C_{l}$ depends only on $n$, $p$ and the norms appearing in
\eqref{eq:data}, whereas $C_{l}'$ may depend additionally on $q$.  Neither
constant depends on $\varrho$, on the geometry of $\Omega$, or on its measure
apart from the values of the prescribed data norms.  The constants $C_{u}$ and $C_{u}'$ of the upper
bounds depend in addition on $|\Omega|$ and on $\|f\|_{p+1}$, while the
threshold $\varrho_{0}$ depends also on $B_{*}$, $\|\nabla f\|_{2}$, $\|f\|_{2}$
and $\|g\|_{2}$.  None of the four constants depends on $\varrho$.  The two
regimes of \eqref{eq:two-sided-sub} and \eqref{eq:two-sided-over} match at the
threshold: the constant $C_{l}'$ does not degenerate as $q$ decreases to
$q_{c}$, as is checked in Remarks \ref{rem:proof} $(iii)$.

\begin{remarks}\label{rem:main}
\begin{enumerate}
\item[$(i)$] The condition $p<\frac{n}{n-2}$ for $n\geq3$ comes from the
Sobolev embedding, while $q<p$ is the source-dominated range, in which
finite-time blow-up occurs under suitable conditions on the initial data such
as \eqref{eq:hyp-blowup}; when $q\geq p$ the energy solutions are global, see
\cite{Georgiev1994,Messaoudi2001}.
\item[$(ii)$] The exponent $\vartheta(p,q)$ is the one predicted by a dominant
balance in the ordinary differential equation
$\phi''=\phi|\phi|^{p-1}-\phi'|\phi'|^{q-1}$, in which the Laplacian is
discarded.  If one looks for a solution behaving like
$\phi\sim c\,(T^{*}-t)^{-\alpha}$, the balance $\phi''\sim\phi^{p}$ gives
$\alpha=2/(p-1)$ and dominates when $q<q_{c}$, whereas the balance
$\phi'|\phi'|^{q-1}\sim\phi^{p}$ gives $\alpha=q/(p-q)$ and dominates when
$q>q_{c}$; at $q=q_{c}$ the two values of $\alpha$ coincide and the three
terms of the equation have the same formal order.  In all cases
$\alpha=1/\vartheta(p,q)$, so that the
amplitude $\varrho\sim(T^{*})^{-\alpha}$ of the datum and the lifespan are
related by $T^{*}\sim\varrho^{-\vartheta(p,q)}$, which is \eqref{eq:two-sided}.
This heuristic plays no role in the proofs, and no blow-up rate of the above
form is established here; it is recorded only to explain why the threshold
$q_{c}$ and the two branches of $\vartheta$ arise.
\item[$(iii)$] The hypotheses \eqref{eq:data} are the natural ones for the
argument of Section \ref{sec:proof}.  The rescaling \eqref{eq:scaling} leaves
the norms of $\dot{W}^{1,n}$ and of $L^{n}$ invariant, so that the assumptions
on $\nabla f$ and on $g$ are critical or better, while the assumption
$f\in L^{\infty}$ is what makes the amplitude of the rescaled profile bounded.
We do not claim that they are necessary.
\item[$(iv)$] No condition ensuring blow-up is used in Theorem
\ref{thm:lower}.  The conditions \eqref{eq:hyp-blowup} enter only through the
upper bounds \eqref{eq:upper}, and are automatically satisfied for all
sufficiently large $\varrho$ when $f\not\equiv0$.
\item[$(v)$] The Lipschitz regularity of $\partial\Omega$ is nowhere used
quantitatively below.  Every embedding is applied to the extension by zero of
a function of $H_{0}^{1}$, which belongs to $H^{1}(\mathbb{R}^{n})$ with the
same norm, and the covering argument of Lemma \ref{lem:glue} uses only the
boundedness of $\Omega$.  We keep the hypothesis because the local existence
theory quoted in Lemma \ref{lem:aux} is usually stated in that setting.
\item[$(vi)$] It would be interesting to extend Theorem \ref{thm:two-sided} to
the equations with variable exponents studied in \cite{Messaoudi2017}, where
$p$ and $q$ depend on $x$; we come back to this point in Section
\ref{sec:open}.
\end{enumerate}
\end{remarks}

The proofs of Theorems \ref{thm:lower} and \ref{thm:two-sided} both take place
in the rescaled frame.  The next section gathers the tools which are needed
there, and the proofs themselves are given in Section \ref{sec:proof}.

\section{Auxiliary results}\label{sec:aux}

To prove the main results, we first need to establish some auxiliary results.
They are of four kinds.  We begin with the change of variables
\eqref{eq:scaling} and with the estimate showing that the rescaled data are of
order one in uniformly local norms.  We then establish two localized energy inequalities for weak
solutions, from which we deduce uniqueness and finite propagation.  We next
construct solutions on a time interval whose length is governed by the
uniformly local energy of the data alone, and not by the damping coefficient.
We finally prove that, when the damping coefficient is large, the uniformly
local energy of the solution can only grow slowly.

\subsection{The rescaled problem}\label{sec:rescaling}

\begin{lemma}\label{lem:rescale}
Let $\lambda=\varrho^{(p-1)/2}$, let $\Omega_{\lambda}=\lambda\Omega$ and let
$v(y,s)=\varrho^{-1}u(y/\lambda,s/\lambda)$.  Then $u$ is a weak solution of
\eqref{eq:main} on $[0,T)$ if and only if $v$ is a weak solution of
\begin{equation}\label{eq:rescaled-pb}
v_{ss}-\Delta v=v|v|^{p-1}-\varepsilon_{\varrho}\,v_{s}|v_{s}|^{q-1}
\quad\text{in }\Omega_{\lambda}\times(0,\lambda T),
\qquad v|_{\partial\Omega_{\lambda}}=0,
\end{equation}
with initial data
\begin{equation}\label{eq:rescaled-data}
v(y,0)=f_{\lambda}(y),
\qquad
v_{s}(y,0)=\lambda^{-1}g_{\lambda}(y),
\qquad\text{where }
f_{\lambda}(y):=f\Bigl(\frac{y}{\lambda}\Bigr),
\ \ g_{\lambda}(y):=g\Bigl(\frac{y}{\lambda}\Bigr),
\end{equation}
and $\varepsilon_{\varrho}=\varrho^{\kappa}$ as in \eqref{eq:params}.  Moreover
\begin{equation}\label{eq:times}
T^{*}(\varrho)=\lambda^{-1}S^{*}_{\lambda},
\end{equation}
where $S^{*}_{\lambda}$ denotes the maximal existence time of the solution of
\eqref{eq:rescaled-pb}--\eqref{eq:rescaled-data}.
\end{lemma}

\begin{proof}
Differentiating \eqref{eq:scaling} gives $u_{tt}=\varrho\lambda^{2}v_{ss}$ and
$\Delta u=\varrho\lambda^{2}\Delta v$, while
$u|u|^{p-1}=\varrho^{p}v|v|^{p-1}$ and
$u_{t}|u_{t}|^{q-1}=(\varrho\lambda)^{q}v_{s}|v_{s}|^{q-1}$.  Since
$\varrho\lambda^{2}=\varrho^{p}$ by the definition of $\lambda$, dividing the
equation by $\varrho^{p}$ produces \eqref{eq:rescaled-pb} with the coefficient
\[
\frac{(\varrho\lambda)^{q}}{\varrho^{p}}
=\varrho^{q(p+1)/2-p}=\varrho^{\kappa}=\varepsilon_{\varrho}.
\]
The data transform as in \eqref{eq:rescaled-data}, and the relation between the
maximal times is immediate from \eqref{eq:scaling}.  Finally, the class
\eqref{eq:class} and the weak formulation \eqref{eq:weak} are preserved under
the space--time change of variables $(x,t)=(y/\lambda,s/\lambda)$, test
functions being transported in the same way.
\end{proof}

\begin{lemma}\label{lem:data}
Assume \eqref{eq:data}.  For every $R\geq1$ there exists a constant $K_{R}>0$,
independent of $\lambda\geq1$, such that
\begin{equation}\label{eq:unifloc-data}
\sup_{z\in\mathbb{R}^{n}}
\Bigl(\|f_{\lambda}\|_{H^{1}(B_{R}(z)\cap\Omega_{\lambda})}
+\|\lambda^{-1}g_{\lambda}\|_{L^{2}(B_{R}(z)\cap\Omega_{\lambda})}\Bigr)
\leq K_{R}.
\end{equation}
\end{lemma}

\begin{proof}
The contribution of $f_{\lambda}$ itself is bounded by
$|B_{R}|^{1/2}\|f\|_{\infty}$.  Let $m=n\vee2$.  Since
$\nabla f_{\lambda}(y)=\lambda^{-1}(\nabla f)(y/\lambda)$, the change of
variables $x=y/\lambda$ followed by H\"older's inequality gives
\[
\|\nabla f_{\lambda}\|_{L^{2}(B_{R}(z)\cap\Omega_{\lambda})}^{2}
\leq\lambda^{n-2}\int_{B_{R/\lambda}(z/\lambda)}|\nabla f(x)|^{2}dx
\leq\lambda^{n-2}\|\nabla f\|_{m}^{2}\,
\bigl|B_{R/\lambda}\bigr|^{1-2/m}.
\]
If $n\geq3$ then $m=n$ and
$|B_{R/\lambda}|^{1-2/n}=c_{n}R^{n-2}\lambda^{-(n-2)}$, which cancels the factor
$\lambda^{n-2}$; if $n\leq2$ then $m=2$ and $\lambda^{n-2}\leq1$.  The same
computation, without the gradient, applies to $\lambda^{-1}g_{\lambda}$.
\end{proof}

Lemma \ref{lem:data} says that the quantity which limits the energy argument
of \cite{BHK}, namely $\|\varrho f\|_{H^{1}(\Omega)}\sim\varrho$, becomes of
order one after the dilation, provided that it is measured on balls of unit
size rather than on the whole domain, that is, in uniformly local norms in the
sense of \cite{Kato1975}.  To turn this into a lower bound for the lifespan we
need a local existence theory in which the existence time is controlled by
uniformly local norms only, and in which the damping coefficient plays no
role.  We first establish the localized energy estimates which such a theory
requires.

\subsection{Localized energy estimates, uniqueness and finite propagation}\label{sec:locid}

The identity of Lemma \ref{lem:loc-energy} cannot be obtained by testing
\eqref{eq:weak} with $\chi^{2}u_{s}$, since $u_{s}(s)$ belongs to $L^{2}(D)$
only, and is therefore not an admissible test function.  The classical device
is to test with the symmetric difference quotient of $u$, which does belong to
$H_{0}^{1}(D)$, and to observe that the two quadratic terms then telescope.  We
state the result for an equation with a general right-hand side, so that both
\eqref{eq:eqD} and the auxiliary problem \eqref{eq:aux-pb} below are covered.
In Lemmas \ref{lem:loc-energy}, \ref{lem:loc-diff} and \ref{lem:finiteprop} the
cutoff $\chi$ is an arbitrary Lipschitz function with the stated properties;
from \eqref{eq:cutoff} onwards it denotes one fixed such function.

\begin{lemma}\label{lem:loc-energy}
Let $\varepsilon>0$, let $h\in C([0,T];L^{2}(D))$, and let $w$ belong to the
class \eqref{eq:class} and satisfy \eqref{eq:weak} with $u|u|^{p-1}$ replaced by
$h$, that is, let $w$ be a weak solution of
\begin{equation}\label{eq:aux-eq}
w_{ss}-\Delta w+\varepsilon\,w_{s}|w_{s}|^{q-1}=h
\quad\text{in }D\times(0,T),
\qquad w|_{\partial D}=0.
\end{equation}
Then, for every Lipschitz function $\chi:\mathbb{R}^{n}\to[0,1]$ with bounded
support and for all $0\leq s_{1}\leq s_{2}\leq T$,
\begin{equation}\label{eq:loc-energy}
\begin{aligned}
&\frac{1}{2}\int_{D}\chi^{2}\bigl(w_{s}^{2}+|\nabla w|^{2}\bigr)(s_{2})\,dy
+\varepsilon\int_{s_{1}}^{s_{2}}\!\!\int_{D}\chi^{2}|w_{s}|^{q+1}dy\,ds\\
&\qquad=\frac{1}{2}\int_{D}\chi^{2}\bigl(w_{s}^{2}+|\nabla w|^{2}\bigr)(s_{1})\,dy
+\int_{s_{1}}^{s_{2}}\!\!\int_{D}\chi^{2}h\,w_{s}\,dy\,ds
-2\int_{s_{1}}^{s_{2}}\!\!\int_{D}\chi\,w_{s}\,\nabla\chi\cdot\nabla w\,dy\,ds.
\end{aligned}
\end{equation}
\end{lemma}

\begin{proof}
Assume first that $0<s_{1}<s_{2}<T$, and let $0<\eta<\min\{s_{1},T-s_{2}\}$.
Define
\[
\delta_{\eta}w(s)=\frac{w(s+\eta)-w(s-\eta)}{2\eta}
=\frac{1}{2\eta}\int_{s-\eta}^{s+\eta}w_{s}(\tau)\,d\tau,
\qquad s\in[s_{1},s_{2}].
\]
Since $w$ belongs to the class \eqref{eq:class}, the function
$\varphi=\chi^{2}\delta_{\eta}w$ belongs to
$C([s_{1},s_{2}];H_{0}^{1}(D))\cap C^{1}([s_{1},s_{2}];L^{2}(D))$ and satisfies
$\varphi_{s}=\chi^{2}\delta_{\eta}w_{s}$, so that it is admissible in
\eqref{eq:weak}.  This gives
\begin{equation}\label{eq:weak-dq}
\Bigl[\int_{D}\chi^{2}w_{s}\,\delta_{\eta}w\Bigr]_{s_{1}}^{s_{2}}
+\int_{s_{1}}^{s_{2}}\!\!\int_{D}
\Bigl[-\chi^{2}w_{s}\,\delta_{\eta}w_{s}
+\nabla w\cdot\nabla\bigl(\chi^{2}\delta_{\eta}w\bigr)
+\varepsilon\,w_{s}|w_{s}|^{q-1}\chi^{2}\delta_{\eta}w
-h\,\chi^{2}\delta_{\eta}w\Bigr]=0.
\end{equation}
We pass to the limit as $\eta$ tends to zero in each term of \eqref{eq:weak-dq}.

Consider first the two quadratic terms, which telescope.  Put
\[
\Theta_{\eta}(s)=\int_{D}\chi^{2}w_{s}(s)\,w_{s}(s+\eta)\,dy,
\]
so that
\[
\int_{D}\chi^{2}w_{s}(s)\,\delta_{\eta}w_{s}(s)\,dy
=\frac{1}{2\eta}\bigl(\Theta_{\eta}(s)-\Theta_{\eta}(s-\eta)\bigr),
\]
and therefore
\[
\int_{s_{1}}^{s_{2}}\!\!\int_{D}\chi^{2}w_{s}\,\delta_{\eta}w_{s}
=\frac{1}{2\eta}\Bigl[\int_{s_{2}-\eta}^{s_{2}}\Theta_{\eta}
-\int_{s_{1}-\eta}^{s_{1}}\Theta_{\eta}\Bigr]
\longrightarrow
\frac{1}{2}\Bigl[\int_{D}\chi^{2}w_{s}^{2}\Bigr]_{s_{1}}^{s_{2}}.
\]
Here we used that $s\mapsto w_{s}(s)$ is uniformly continuous with values in
$L^{2}(D)$, whence
$|\Theta_{\eta}(s)-\int_{D}\chi^{2}w_{s}^{2}(s)|
\leq\|w_{s}(s)\|_{2}\|w_{s}(s+\eta)-w_{s}(s)\|_{2}$ tends to zero uniformly in
$s$, together with the continuity of $s\mapsto\int_{D}\chi^{2}w_{s}^{2}$.  The
same computation with
$\Xi_{\eta}(s)=\int_{D}\chi^{2}\nabla w(s)\cdot\nabla w(s+\eta)\,dy$ and with
$w\in C([0,T];H_{0}^{1}(D))$ gives
\[
\int_{s_{1}}^{s_{2}}\!\!\int_{D}\chi^{2}\nabla w\cdot\nabla\delta_{\eta}w
\longrightarrow
\frac{1}{2}\Bigl[\int_{D}\chi^{2}|\nabla w|^{2}\Bigr]_{s_{1}}^{s_{2}}.
\]

Consider next the remaining terms.  Since $w\in C^{1}([0,T];L^{2}(D))$, we have
$\delta_{\eta}w\to w_{s}$ in $C([s_{1},s_{2}];L^{2}(D))$, and since
$\delta_{\eta}w$ is the average of $w_{s}$ over $[s-\eta,s+\eta]$ and
$w_{s}\in L^{q+1}$, the continuity of translations in
$L^{q+1}(D\times(s_{1},s_{2}))$ gives $\delta_{\eta}w\to w_{s}$ in that space
as well.  Consequently
\[
\Bigl[\int_{D}\chi^{2}w_{s}\,\delta_{\eta}w\Bigr]_{s_{1}}^{s_{2}}
\longrightarrow\Bigl[\int_{D}\chi^{2}w_{s}^{2}\Bigr]_{s_{1}}^{s_{2}},
\qquad
\int_{s_{1}}^{s_{2}}\!\!\int_{D}h\,\chi^{2}\delta_{\eta}w
\longrightarrow\int_{s_{1}}^{s_{2}}\!\!\int_{D}\chi^{2}h\,w_{s},
\]
and, by duality between $L^{(q+1)/q}$ and $L^{q+1}$,
\[
\varepsilon\int_{s_{1}}^{s_{2}}\!\!\int_{D}w_{s}|w_{s}|^{q-1}\chi^{2}\delta_{\eta}w
\longrightarrow
\varepsilon\int_{s_{1}}^{s_{2}}\!\!\int_{D}\chi^{2}|w_{s}|^{q+1}.
\]
Finally, writing
\[
\nabla w\cdot\nabla\bigl(\chi^{2}\delta_{\eta}w\bigr)
=\chi^{2}\nabla w\cdot\nabla\delta_{\eta}w
+2\chi\,\delta_{\eta}w\,\nabla\chi\cdot\nabla w,
\]
the second term converges to
$2\int\!\!\int\chi\,w_{s}\,\nabla\chi\cdot\nabla w$.

Denoting by $\Delta_{1}$ and $\Delta_{2}$ the increments of
$\int_{D}\chi^{2}w_{s}^{2}$ and of $\int_{D}\chi^{2}|\nabla w|^{2}$ between
$s_{1}$ and $s_{2}$, the limit of \eqref{eq:weak-dq} reads
\[
\Delta_{1}-\frac{1}{2}\Delta_{1}+\frac{1}{2}\Delta_{2}
+2\int\!\!\int\chi\,w_{s}\,\nabla\chi\cdot\nabla w
+\varepsilon\int\!\!\int\chi^{2}|w_{s}|^{q+1}
-\int\!\!\int\chi^{2}h\,w_{s}=0,
\]
which is \eqref{eq:loc-energy}.  Both sides of \eqref{eq:loc-energy} are
continuous functions of $(s_{1},s_{2})$, the boundary terms by
\eqref{eq:class} and the time integrals because their integrands belong to
$L^{1}(0,T)$, so the restriction $0<s_{1}<s_{2}<T$ may be removed.
\end{proof}

\begin{lemma}\label{lem:loc-diff}
Let $\varepsilon>0$, let $h^{1},h^{2}\in C([0,T];L^{2}(D))$, and let $w^{1}$ and
$w^{2}$ be weak solutions of \eqref{eq:aux-eq} with right-hand sides $h^{1}$ and
$h^{2}$ respectively, both in the class \eqref{eq:class}.  Then $z=w^{1}-w^{2}$
satisfies, for every Lipschitz function $\chi:\mathbb{R}^{n}\to[0,1]$ with
bounded support and for all $0\leq s_{1}\leq s_{2}\leq T$,
\begin{equation}\label{eq:loc-diff}
\begin{aligned}
&\frac{1}{2}\int_{D}\chi^{2}\bigl(z_{s}^{2}+|\nabla z|^{2}\bigr)(s_{2})\,dy
\leq\frac{1}{2}\int_{D}\chi^{2}\bigl(z_{s}^{2}+|\nabla z|^{2}\bigr)(s_{1})\,dy\\
&\qquad\qquad
+\int_{s_{1}}^{s_{2}}\!\!\int_{D}\chi^{2}\bigl(h^{1}-h^{2}\bigr)z_{s}\,dy\,ds
-2\int_{s_{1}}^{s_{2}}\!\!\int_{D}\chi\,z_{s}\,\nabla\chi\cdot\nabla z\,dy\,ds.
\end{aligned}
\end{equation}
\end{lemma}

\begin{proof}
Subtract the two identities \eqref{eq:weak} satisfied by $w^{1}$ and $w^{2}$ and
test the result with $\chi^{2}\delta_{\eta}z$.  Every step of the proof of
Lemma \ref{lem:loc-energy} applies without change, with $h^{1}-h^{2}$ in place
of $h$, and yields in the limit the identity \eqref{eq:loc-diff} with the
additional term
\[
\varepsilon\int_{s_{1}}^{s_{2}}\!\!\int_{D}\chi^{2}
\bigl(w_{s}^{1}|w_{s}^{1}|^{q-1}-w_{s}^{2}|w_{s}^{2}|^{q-1}\bigr)
\bigl(w_{s}^{1}-w_{s}^{2}\bigr)dy\,ds
\]
on the left-hand side, the convergence of the damping term following as before
from $w^{i}_{s}|w^{i}_{s}|^{q-1}\in L^{(q+1)/q}$ and from
$\delta_{\eta}z\to z_{s}$ in $L^{q+1}$.  The map
$\zeta\mapsto\zeta|\zeta|^{q-1}$ being nondecreasing, that term is nonnegative,
and discarding it gives \eqref{eq:loc-diff}.
\end{proof}

The inequalities \eqref{eq:loc-energy} and \eqref{eq:loc-diff} will always be
used in combination with the following elementary identity: for every bounded
measurable $\chi:\mathbb{R}^{n}\to\mathbb{R}$ and every
$z\in C^{1}([0,T];L^{2}(D))$, the function
$s\mapsto\int_{D}\chi^{2}z^{2}(s)\,dy$ is of class $C^{1}$ with derivative
$2\int_{D}\chi^{2}z\,z_{s}\,dy$, so that
\begin{equation}\label{eq:L2loc}
\frac{1}{2}\int_{D}\chi^{2}z^{2}(s_{2})\,dy
=\frac{1}{2}\int_{D}\chi^{2}z^{2}(s_{1})\,dy
+\int_{s_{1}}^{s_{2}}\!\!\int_{D}\chi^{2}z\,z_{s}\,dy\,ds,
\qquad 0\leq s_{1}\leq s_{2}\leq T.
\end{equation}
The first consequence of Lemma \ref{lem:loc-diff} is uniqueness, which makes
the maximal existence time well defined.

\begin{lemma}\label{lem:unique}
Under \eqref{eq:exponents}, two weak solutions of \eqref{eq:eqD} in the class
\eqref{eq:class} with the same initial data coincide.  The same holds for
\eqref{eq:aux-eq} with a given right-hand side $h$.
\end{lemma}

\begin{proof}
Let $u^{1},u^{2}$ be two such solutions of \eqref{eq:eqD} and $z=u^{1}-u^{2}$.
Since $u^{i}\in C([0,T];H_{0}^{1}(D))$ and
$H_{0}^{1}(D)\hookrightarrow L^{2p}(D)$ under \eqref{eq:exponents}, the functions
$h^{i}=u^{i}|u^{i}|^{p-1}$ belong to $C([0,T];L^{2}(D))$, so that Lemma
\ref{lem:loc-diff} applies.  Choose $\chi$ equal to one on a ball containing
$\overline{D}$, so that $\nabla\chi=0$ on $D$, and add to \eqref{eq:loc-diff}
the identity \eqref{eq:L2loc} with this $\chi$.  Setting
\[
I(s)=\frac{1}{2}\int_{D}\bigl(z_{s}^{2}+|\nabla z|^{2}+z^{2}\bigr)(s)\,dy,
\]
and using the mean value theorem in the form
$|h^{1}-h^{2}|\leq p(|u^{1}|+|u^{2}|)^{p-1}|z|$ together with H\"older's
inequality with the exponents given by
$\frac{1}{2}=\frac{p-1}{2p}+\frac{1}{2p}$ and the embedding
$H_{0}^{1}(D)\hookrightarrow L^{2p}(D)$, we obtain
\[
I(s_{2})\leq I(s_{1})+C\int_{s_{1}}^{s_{2}}I(\tau)\,d\tau,
\qquad 0\leq s_{1}\leq s_{2}\leq T,
\]
with a constant depending on $\sup_{s}\|u^{i}(s)\|_{2p}$.  Since $I(0)=0$, the
integral form of Gronwall's lemma gives $I\equiv0$.  For \eqref{eq:aux-eq} the
same argument applies with $h^{1}-h^{2}=0$.
\end{proof}

The existence part of the local theory will be obtained by a fixed point
argument, in which the source term is frozen.  The corresponding auxiliary
problem is the following.

\begin{lemma}\label{lem:aux}
Let $D$ be a bounded Lipschitz domain, let $\varepsilon>0$, let
$(w_{0},w_{1})\in H_{0}^{1}(D)\times L^{2}(D)$ and let
$h\in C([0,T];L^{2}(D))$.  Then the problem
\begin{equation}\label{eq:aux-pb}
w_{ss}-\Delta w+\varepsilon\,w_{s}|w_{s}|^{q-1}=h,
\qquad w|_{\partial D}=0,
\qquad (w,w_{s})|_{s=0}=(w_{0},w_{1}),
\end{equation}
has a unique weak solution in the class \eqref{eq:class}, and this solution
satisfies the energy identity
\begin{equation}\label{eq:energy-id}
\frac{1}{2}\|w_{s}(t)\|_{2}^{2}+\frac{1}{2}\|\nabla w(t)\|_{2}^{2}
+\varepsilon\int_{0}^{t}\!\!\int_{D}|w_{s}|^{q+1}
=\frac{1}{2}\|w_{1}\|_{2}^{2}+\frac{1}{2}\|\nabla w_{0}\|_{2}^{2}
+\int_{0}^{t}\!\!\int_{D}h\,w_{s}.
\end{equation}
\end{lemma}

\begin{proof}
Existence is the Faedo--Galerkin construction for a wave equation with a
monotone damping term, for which we refer to \cite[Chapter 1]{Lions} and to
\cite[Section 2]{Georgiev1994}; the damping term is identified in the limit by
monotonicity, as in \cite[Section 2]{BociuLasiecka2010}.  We recall that the strong
continuity required in \eqref{eq:class} is not a consequence of linear theory
here, since $w_{s}|w_{s}|^{q-1}$ belongs to $L^{(q+1)/q}$ only and
$(q+1)/q<2$; it is obtained in the quoted references in the standard way, by
combining the weak continuity of $s\mapsto(w(s),w_{s}(s))$ in
$H_{0}^{1}(D)\times L^{2}(D)$ with the continuity in $s$ of
$\|w_{s}(s)\|_{2}^{2}+\|\nabla w(s)\|_{2}^{2}$, which the energy identity
provides.  Uniqueness is contained in Lemma \ref{lem:unique}.  Once the
solution is known to belong to \eqref{eq:class}, the identity
\eqref{eq:energy-id} is \eqref{eq:loc-energy} with $s_{1}=0$, $s_{2}=t$ and
$\chi$ equal to one on a ball containing $\overline{D}$, so that $\nabla\chi=0$
on $D$.
\end{proof}

The two inequalities \eqref{eq:loc-energy} and \eqref{eq:loc-diff} are the only
properties of weak solutions that will be used in the sequel.  The damping
coefficient $\varepsilon$ appears in \eqref{eq:loc-energy} on the left-hand
side, hence with a favourable sign, and does not appear in \eqref{eq:loc-diff}
at all, having been discarded together with the monotone term.  This is the
reason why all the constants obtained below are independent of $\varepsilon$.
We now prove that solutions propagate at a finite speed.  Instead of using a
cutoff function adapted to a light cone, whose gradient becomes unbounded as
the cone is approximated, we work with a family of nested cutoff functions of
fixed width and let the number of steps tend to infinity.  In this way no
derivative of a cutoff function ever has to be controlled by the localized
energy itself.  The speed obtained is finite but is not equal to one, which is
sufficient for the applications made below.

\begin{lemma}\label{lem:finiteprop}
Let $D$ be a bounded Lipschitz domain, let $\varepsilon>0$, and let $u^{1}$ and
$u^{2}$ be weak solutions of \eqref{eq:eqD} on $D\times[0,T]$ in the class
\eqref{eq:class} whose initial data coincide almost everywhere on
$B_{R}(y_{0})\cap D$, where $y_{0}\in\mathbb{R}^{n}$ and $0<r_{0}<R\leq10$; only
the case $R=1$ and $r=r_{0}=\frac{1}{2}$ is used below, in the proof of Lemma
\ref{lem:glue}, and the restriction $R\leq10$ is a normalization which plays no
further role.  Put
\[
M=\max_{i=1,2}\ \sup_{0\leq s\leq T}\|u^{i}(s)\|_{L^{2p}(D)}<\infty.
\]
Then there exists a constant $\sigma=\sigma(M,n,p,r_{0})<\infty$, nondecreasing
in $M$, such that for every $r\in[r_{0},R)$,
\begin{equation}\label{eq:finiteprop}
u^{1}=u^{2}
\quad\text{almost everywhere in }
\Bigl\{(y,s):\ 0\leq s\leq\min\Bigl(T,\frac{R-r}{\sigma}\Bigr),\
y\in B_{r}(y_{0})\cap D\Bigr\}.
\end{equation}
\end{lemma}

\begin{proof}
Let $z=u^{1}-u^{2}$ and, for $0<\rho\leq R$, set
\[
I_{\rho}(s)=\frac{1}{2}\int_{B_{\rho}(y_{0})\cap D}
\bigl(z_{s}^{2}+|\nabla z|^{2}+z^{2}\bigr)(y,s)\,dy,
\]
which is a continuous function of $s$ satisfying
$I^{*}:=\sup_{[0,T]}I_{R}<\infty$ and $I_{\rho}(0)=0$ for $\rho\leq R$.  All the
radii occurring below belong to $[r_{0},R]$.

\emph{Step 1: an inequality between nested radii.}  Let
$r_{0}\leq\rho<\rho+\delta\leq R$ and let $\chi$ be a smooth function equal to
one on $B_{\rho}(y_{0})$, supported in $B_{\rho+\delta}(y_{0})$, with values in
$[0,1]$ and with $|\nabla\chi|\leq2/\delta$.  Adding \eqref{eq:L2loc} to
\eqref{eq:loc-diff}, using $\chi^{2}\leq\mathbf{1}_{B_{\rho+\delta}}$,
$I_{\rho}(s)\leq\frac{1}{2}\int\chi^{2}(z_{s}^{2}+|\nabla z|^{2}+z^{2})(s)$ and
$I_{\rho+\delta}(0)=0$, we obtain
\begin{equation}\label{eq:Irho}
I_{\rho}(s)\leq\int_{0}^{s}\bigl[J_{1}(\tau)+J_{2}(\tau)+J_{3}(\tau)\bigr]d\tau,
\end{equation}
where, all integrals being taken over $B_{\rho+\delta}(y_{0})\cap D$,
\[
J_{1}=\int\bigl|u^{1}|u^{1}|^{p-1}-u^{2}|u^{2}|^{p-1}\bigr||z_{s}|,
\qquad
J_{2}=\int|z||z_{s}|,
\qquad
J_{3}=\frac{4}{\delta}\int|z_{s}||\nabla z|.
\]
By the mean value theorem,
$|u^{1}|u^{1}|^{p-1}-u^{2}|u^{2}|^{p-1}|\leq p(|u^{1}|+|u^{2}|)^{p-1}|z|$, so
that H\"older's inequality with the exponents given by
$\frac{1}{2}=\frac{p-1}{2p}+\frac{1}{2p}$, together with
$\||u^{1}|+|u^{2}|\|_{2p}^{p-1}\leq(2M)^{p-1}$, yields
\[
J_{1}\leq p\,2^{p-1}M^{p-1}
\|z\|_{L^{2p}(B_{\rho+\delta}\cap D)}
\|z_{s}\|_{L^{2}(B_{\rho+\delta}\cap D)}.
\]
Since $z(\tau)\in H_{0}^{1}(D)$, its extension by zero, which we still denote
by $z$, belongs to $H^{1}(\mathbb{R}^{n})$ with the same norm, and vanishes on
$B_{\rho+\delta}\setminus D$; applying the Sobolev embedding on the ball
$B_{\rho+\delta}$ to this extension therefore gives
\begin{equation}\label{eq:sob-ball}
\|z\|_{L^{2p}(B_{\rho+\delta}\cap D)}
\leq C_{1}\|z\|_{H^{1}(B_{\rho+\delta}\cap D)}
\leq C_{1}\bigl(2I_{\rho+\delta}\bigr)^{1/2},
\end{equation}
where $C_{1}=C_{1}(n,p,r_{0})$ is the supremum of the embedding constants of
$H^{1}(B_{r'})\hookrightarrow L^{2p}(B_{r'})$ for $r'\in[r_{0},10]$.  This
supremum is finite: the scaling $\tilde{z}(x)=z(r'x)$ gives
\[
\|z\|_{L^{2p}(B_{r'})}
\leq C(n,p)\,(r')^{-\frac{n}{2}\left(1-\frac{1}{p}\right)}
\Bigl(r'\|\nabla z\|_{L^{2}(B_{r'})}
+\|z\|_{L^{2}(B_{r'})}\Bigr),
\]
which is bounded uniformly for $r'$ in a compact subset of $(0,\infty)$,
but not as $r'$ tends to zero.  This is the only point of the proof at
which the lower bound $r_{0}$ is needed.  Hence
$J_{1}\leq C_{2}(M)I_{\rho+\delta}$ with $C_{2}(M)=2^{p}pC_{1}M^{p-1}$.  Moreover
$J_{2}\leq I_{\rho+\delta}$, and, since
$|z_{s}||\nabla z|\leq\frac{1}{2}(z_{s}^{2}+|\nabla z|^{2})$, also
$J_{3}\leq\frac{4}{\delta}I_{\rho+\delta}$.  Setting
$C_{3}=10\,(C_{2}(M)+1)+4$, which depends only on $M$, $n$, $p$ and $r_{0}$,
and using $\delta\leq R\leq10$, we conclude from \eqref{eq:Irho} that
\begin{equation}\label{eq:nested}
I_{\rho}(s)\leq\frac{C_{3}}{\delta}\int_{0}^{s}I_{\rho+\delta}(\tau)\,d\tau
\qquad\text{whenever }r_{0}\leq\rho<\rho+\delta\leq R.
\end{equation}

\emph{Step 2: iteration.}  Let $r\in[r_{0},R)$ and $k\in\mathbb{N}$, and apply
\eqref{eq:nested} $k$ times with the constant step $\delta=(R-r)/k$, at the
radii $r,r+\delta,\dots,r+(k-1)\delta$, all of which belong to $[r_{0},R]$.
This gives
\[
I_{r}(s)\leq\Bigl(\frac{C_{3}k}{R-r}\Bigr)^{k}
\int_{0}^{s}\!\!\int_{0}^{\tau_{1}}\!\!\cdots\int_{0}^{\tau_{k-1}}
I_{R}(\tau_{k})\,d\tau_{k}\cdots d\tau_{1}
\leq\Bigl(\frac{C_{3}k}{R-r}\Bigr)^{k}\frac{s^{k}}{k!}\,I^{*},
\]
and, since $k!\geq(k/e)^{k}$,
\[
I_{r}(s)\leq\Bigl(\frac{eC_{3}\,s}{R-r}\Bigr)^{k}I^{*}
\qquad\text{for every }k\in\mathbb{N}.
\]
If $s<(R-r)/(eC_{3})$, the quantity in brackets is smaller than one, and
letting $k$ tend to infinity gives $I_{r}(s)=0$; by continuity in $s$ the same
holds at $s=(R-r)/(eC_{3})$.  Therefore $z$ vanishes on $B_{r}(y_{0})\cap D$
for $s\leq\min(T,(R-r)/\sigma)$ with $\sigma=eC_{3}$, which has the stated
dependence.
\end{proof}

The speed $\sigma$ produced by Lemma \ref{lem:finiteprop} depends on the
solutions through the quantity $M$, and is in general larger than one.  This is
of no consequence for what follows, since the only role of finite propagation
here is to guarantee that two solutions constructed from data which agree on a
ball agree on a smaller ball for a time which is bounded below independently of
$\lambda$ and of $\varepsilon$.

\subsection{Local existence in uniformly local norms}\label{sec:localexist}

We can now construct solutions of \eqref{eq:eqD} on a time interval which
depends on the data only through their energy norm, and not on
$\varepsilon$; the passage from the energy norm on $D$ to uniformly local
norms is made in Lemma \ref{lem:glue}.

\begin{lemma}\label{lem:localexist}
Let $N>0$.  There exists a time $T_{\rm loc}=T_{\rm loc}(N;n,p)>0$ such that,
for every bounded Lipschitz domain $D$, every $\varepsilon>0$ and every
$(v_{0},v_{1})\in H_{0}^{1}(D)\times L^{2}(D)$ with
$\|v_{0}\|_{H^{1}}+\|v_{1}\|_{2}\leq N$, problem \eqref{eq:eqD} with data
$(v_{0},v_{1})$ has a unique weak solution on $[0,T_{\rm loc}]$ in the class
\eqref{eq:class}, and
\begin{equation}\label{eq:loc-bound}
\sup_{0\leq s\leq T_{\rm loc}}
\bigl(\|v(s)\|_{H^{1}}+\|v_{s}(s)\|_{2}\bigr)\leq3N.
\end{equation}
The time $T_{\rm loc}$ does not depend on $D$, on $\varepsilon$ or on $q$.
\end{lemma}

\begin{proof}
Let
\[
X_{T}=\Bigl\{u\in C([0,T];H_{0}^{1}(D))\cap C^{1}([0,T];L^{2}(D)):\
\|u\|_{X_{T}}:=\sup_{[0,T]}\bigl(\|u\|_{H^{1}}+\|u_{s}\|_{2}\bigr)\leq3N\Bigr\},
\]
which is a complete metric space, and for $u\in X_{T}$ let $\mathcal{F}(u)$ be
the solution of \eqref{eq:aux-pb} with right-hand side $h=u|u|^{p-1}$, given by
Lemma \ref{lem:aux}.  Extension by zero and the Sobolev embedding in
$\mathbb{R}^{n}$ give $\|w\|_{L^{2p}(D)}\leq C_{4}\|w\|_{H^{1}}$ for
$w\in H_{0}^{1}(D)$, with $C_{4}=C_{4}(n,p)$ independent of $D$; hence
$\|h\|_{2}\leq\|u\|_{2p}^{p}\leq(3C_{4}N)^{p}$, and $u\mapsto u|u|^{p-1}$ is
continuous from $H_{0}^{1}(D)$ into $L^{2}(D)$, so that
$h\in C([0,T];L^{2}(D))$ as required in Lemma \ref{lem:aux}.

We check that $\mathcal{F}$ maps $X_{T}$ into itself for $T$ small.  Write
$Y(t)=(\|\mathcal{F}(u)_{s}(t)\|_{2}^{2}+\|\nabla\mathcal{F}(u)(t)\|_{2}^{2})^{1/2}$.
The hypothesis $\|v_{0}\|_{H^{1}}+\|v_{1}\|_{2}\leq N$ gives
$Y(0)\leq\|\nabla v_{0}\|_{2}+\|v_{1}\|_{2}\leq N$, and discarding the damping
term in \eqref{eq:energy-id} yields
$\frac{1}{2}Y(t)^{2}\leq\frac{1}{2}Y(0)^{2}+\int_{0}^{t}\|h\|_{2}Y$, whence
$Y(t)\leq N+T(3C_{4}N)^{p}$ on $[0,T]$.  Consequently
\[
\sup_{[0,T]}\bigl(\|\mathcal{F}(u)_{s}\|_{2}+\|\nabla\mathcal{F}(u)\|_{2}\bigr)
\leq\sqrt{2}\bigl(N+T(3C_{4}N)^{p}\bigr),
\qquad
\|\mathcal{F}(u)(t)\|_{2}\leq\|v_{0}\|_{2}+\int_{0}^{t}\|\mathcal{F}(u)_{s}\|_{2},
\]
so that $\|\mathcal{F}(u)\|_{X_{T}}\leq(1+\sqrt{2})N+CT$ with $C=C(N;n,p)$.
Since $1+\sqrt{2}<3$, the map $\mathcal{F}$ sends $X_{T}$ into
itself as soon as $T\leq T_{1}(N;n,p)$.

Let now $u,\bar{u}\in X_{T}$, put $h=u|u|^{p-1}$, $\bar{h}=\bar{u}|\bar{u}|^{p-1}$
and $w=\mathcal{F}(u)-\mathcal{F}(\bar{u})$, and write
$d=\|u-\bar{u}\|_{X_{T}}$.  By the mean value theorem, H\"older's inequality
with the exponents given by $\frac{1}{2}=\frac{p-1}{2p}+\frac{1}{2p}$, and the
embedding $\|\cdot\|_{L^{2p}(D)}\leq C_{4}\|\cdot\|_{H^{1}}$ recalled above,
\begin{equation}\label{eq:h-lip}
\|h-\bar{h}\|_{2}
\leq p\bigl\||u|+|\bar{u}|\bigr\|_{2p}^{p-1}\|u-\bar{u}\|_{2p}
\leq C_{8}\,d,
\qquad
C_{8}=p\,(6C_{4}N)^{p-1}C_{4},
\end{equation}
since $\||u|+|\bar{u}|\|_{2p}\leq2\cdot3C_{4}N$ for $u,\bar{u}\in X_{T}$; the
constant $C_{8}$ depends only on $N$, $n$ and $p$.  The functions
$\mathcal{F}(u)$ and $\mathcal{F}(\bar{u})$ are weak solutions of
\eqref{eq:aux-eq} with the same data $(v_{0},v_{1})$ and with right-hand sides
$h$ and $\bar{h}$, so Lemma \ref{lem:loc-diff}, applied with $\chi$ equal to one
on a ball containing $\overline{D}$ and hence with $\nabla\chi=0$ on $D$, gives
\[
\frac{1}{2}\int_{D}\bigl(w_{s}^{2}+|\nabla w|^{2}\bigr)(t)\,dy
\leq\int_{0}^{t}\!\!\int_{D}(h-\bar{h})\,w_{s}\,dy\,ds,
\qquad 0\leq t\leq T,
\]
the boundary terms at $s=0$ vanishing because $w(0)=w_{s}(0)=0$.  The damping
does not appear in this inequality, having been discarded in Lemma
\ref{lem:loc-diff} together with the monotone term, so that neither
$\varepsilon$ nor $q$ enters the estimate.  Adding \eqref{eq:L2loc} with the same $\chi$, and setting
$\Psi(t)=\bigl(\|w(t)\|_{2}^{2}+\|\nabla w(t)\|_{2}^{2}+\|w_{s}(t)\|_{2}^{2}\bigr)^{1/2}$,
we obtain
\[
\frac{1}{2}\Psi(t)^{2}
\leq\int_{0}^{t}\bigl(\|h-\bar{h}\|_{2}+\|w\|_{2}\bigr)\|w_{s}\|_{2}\,ds
\leq\int_{0}^{t}\bigl(C_{8}d+\Psi\bigr)\Psi\,ds
\]
by \eqref{eq:h-lip}.  Writing $\Psi^{*}=\sup_{[0,T]}\Psi$ and taking the
supremum over $t\in[0,T]$ gives
$\frac{1}{2}(\Psi^{*})^{2}\leq T(C_{8}d+\Psi^{*})\Psi^{*}$, whence
$\Psi^{*}\leq4TC_{8}d$ as soon as $T\leq\frac{1}{4}$.  Since
$\|w\|_{X_{T}}\leq\sqrt{2}\,\Psi^{*}$, this reads
\[
\|\mathcal{F}(u)-\mathcal{F}(\bar{u})\|_{X_{T}}
\leq4\sqrt{2}\,C_{8}\,T\,\|u-\bar{u}\|_{X_{T}},
\]
which is a contraction for $T\leq T_{2}(N;n,p)$.  Banach's fixed point theorem
on $[0,T_{\rm loc}]$ with $T_{\rm loc}=\min\{T_{1},T_{2}\}$ produces a solution
$v$; it solves \eqref{eq:aux-pb} with $h=v|v|^{p-1}$, so that
$v_{s}\in L^{q+1}$ by \eqref{eq:energy-id}, and it is unique by Lemma
\ref{lem:unique}.  None of the constants above involves the size of
$\varepsilon$.
\end{proof}

Lemma \ref{lem:localexist} controls the existence time by the full energy norm
of the data on $D$.  On the dilated domain $\Omega_{\lambda}$ this norm grows
with $\lambda$, and the lemma alone would give an existence time tending to
zero.  Finite propagation is what allows us to replace the full norm by the
uniformly local one.

\begin{lemma}\label{lem:glue}
Let $\lambda\geq1$, let $\varepsilon>0$ and let
$(v_{0},v_{1})\in H_{0}^{1}(\Omega_{\lambda})\times L^{2}(\Omega_{\lambda})$
satisfy
\begin{equation}\label{eq:unifloc-hyp}
\sup_{z\in\mathbb{R}^{n}}
\Bigl(\|v_{0}\|_{H^{1}(B_{8}(z)\cap\Omega_{\lambda})}
+\|v_{1}\|_{L^{2}(B_{8}(z)\cap\Omega_{\lambda})}\Bigr)\leq N.
\end{equation}
Then there exists $S_{1}=S_{1}(N;n,p)>0$, independent of $\lambda$, of
$\varepsilon$ and of $q$, such that problem \eqref{eq:eqD} on
$D=\Omega_{\lambda}$, with damping coefficient $\varepsilon$ and data
$(v_{0},v_{1})$, has a unique weak solution on $[0,S_{1}]$ in the class
\eqref{eq:class}, and
\begin{equation}\label{eq:unifloc-bound}
\sup_{z\in\mathbb{R}^{n}}\ \sup_{0\leq s\leq S_{1}}
\Bigl(\|v(s)\|_{H^{1}(B_{6}(z)\cap\Omega_{\lambda})}
+\|v_{s}(s)\|_{L^{2}(B_{6}(z)\cap\Omega_{\lambda})}\Bigr)\leq C_{5}'N,
\end{equation}
where $C_{5}'$ depends only on $n$.
\end{lemma}

\begin{proof}
Since $\Omega_{\lambda}$ is bounded, we may fix once and for all a maximal
$6$-separated subset $\{z_{j}\}_{j=1}^{J}$ of $\overline{\Omega_{\lambda}}$, so
that the balls $B_{6}(z_{j})$ cover $\overline{\Omega_{\lambda}}$ and every
ball of radius $6$ meets at most $C(n)$ of them, together with a smooth
partition of unity $\{\psi_{j}\}_{j=1}^{J}$ subordinate to this cover; the
number $J$ depends on $\lambda$, but it enters no estimate and is used
qualitatively only.  For $1\leq j\leq J$ let $\xi_{j}$ be a smooth function
equal to one on $B_{7}(z_{j})$, supported in $B_{8}(z_{j})$, with derivatives
bounded independently of $j$.  Then
$(\xi_{j}v_{0},\xi_{j}v_{1})\in H_{0}^{1}(\Omega_{\lambda})\times
L^{2}(\Omega_{\lambda})$, with norms bounded by $C_{5}N$ in view of
\eqref{eq:unifloc-hyp}, where $C_{5}$ depends only on $n$; indeed,
approximating $v_{0}$ by functions
$w_{k}\in C_{c}^{\infty}(\Omega_{\lambda})$ gives
$\xi_{j}w_{k}\to\xi_{j}v_{0}$ in $H^{1}$.  Lemma \ref{lem:localexist},
applied on $\Omega_{\lambda}$ with these data, produces solutions $v^{j}$ on
$[0,T_{\rm loc}(C_{5}N)]$ satisfying \eqref{eq:loc-bound}, with $T_{\rm loc}$
independent of $j$, of $\lambda$, of $\varepsilon$ and of $q$; in particular
$\sup_{s}\|v^{j}(s)\|_{L^{2p}(\Omega_{\lambda})}\leq M_{0}:=3C_{4}C_{5}N$ by the
Sobolev embedding.  Let $\sigma_{0}=\sigma(M_{0},n,p,\frac{1}{2})$ be the speed
given by Lemma \ref{lem:finiteprop} with $r_{0}=\frac{1}{2}$, and set
\[
S_{1}=\min\Bigl\{T_{\rm loc}(C_{5}N;n,p),\ \frac{1}{2\sigma_{0}}\Bigr\}.
\]

Let $1\leq i,j\leq J$ and let $y\in B_{6}(z_{i})\cap B_{6}(z_{j})$.  Then
$B_{1}(y)\subset B_{7}(z_{i})\cap B_{7}(z_{j})$, so the truncated data used to
construct $v^{i}$ and $v^{j}$ both coincide with $(v_{0},v_{1})$ on
$B_{1}(y)\cap\Omega_{\lambda}$; Lemma \ref{lem:finiteprop} with $y_{0}=y$,
$R=1$ and $r=r_{0}=\frac{1}{2}$ then gives $v^{i}=v^{j}$ almost everywhere on
$(B_{1/2}(y)\cap\Omega_{\lambda})\times[0,S_{1}]$, because
$S_{1}\leq\frac{1}{2\sigma_{0}}$.  Covering $B_{6}(z_{i})\cap B_{6}(z_{j})$ by
countably many such balls, we obtain $v^{i}=v^{j}$ almost everywhere on
$(B_{6}(z_{i})\cap B_{6}(z_{j})\cap\Omega_{\lambda})\times[0,S_{1}]$; only the
finitely many pairs $(i,j)$ occur here, so that all the exceptional sets
together form a single set of measure zero.

Set $v=\sum_{j}\psi_{j}v^{j}$.  By the preceding paragraph and
$\sum_{j}\psi_{j}=1$ on $\Omega_{\lambda}$, the function $v$ coincides almost
everywhere with $v^{j}$ on $(B_{6}(z_{j})\cap\Omega_{\lambda})\times[0,S_{1}]$
for every $j$, and it belongs to the class \eqref{eq:class}.  Moreover $v$ is a
weak solution of \eqref{eq:eqD} on $D=\Omega_{\lambda}$: given a test function
$\varphi$ as in \eqref{eq:weak}, each $\psi_{j}\varphi$ is again an admissible
test function, supported in $B_{6}(z_{j})$ where $v=v^{j}$; writing
\eqref{eq:weak} for $v^{j}$ against $\psi_{j}\varphi$ and summing over $j$
gives \eqref{eq:weak} for $v$ against $\varphi$.  The data are attained, since
$\xi_{j}=1$ on $B_{7}(z_{j})\supset B_{6}(z_{j})$, so that
$(v,v_{s})(0)=(v_{0},v_{1})$ on $B_{6}(z_{j})\cap\Omega_{\lambda}$ for every
$j$.  Finally, for an arbitrary $z\in\mathbb{R}^{n}$ the ball $B_{6}(z)$ meets
at most $C(n)$ of the balls $B_{6}(z_{j})$, on each of which $v$ agrees almost
everywhere with $v^{j}$; summing the bounds \eqref{eq:loc-bound} for those
finitely many indices gives \eqref{eq:unifloc-bound} with a constant $C_{5}'$
depending only on $n$.  Uniqueness is Lemma \ref{lem:unique}.
\end{proof}

The solution provided by Lemma \ref{lem:glue} exists on a time interval whose
length depends on the data only through the uniformly local norm
\eqref{eq:unifloc-hyp}, and not on the damping coefficient.  The lemma also
yields a blow-up alternative in the same norm, which is the form in which it
will be used in the overdamped regime.

\begin{lemma}\label{lem:blowupalt}
Let $\lambda\geq1$, let $\varepsilon>0$ and let $v$ be the maximal weak
solution of \eqref{eq:eqD} on $D=\Omega_{\lambda}$, with damping coefficient
$\varepsilon$, in the class \eqref{eq:class}.  If its maximal existence time
$S_{\max}$ is finite, then
\[
\limsup_{s\uparrow S_{\max}}\ \sup_{z\in\mathbb{R}^{n}}
\Bigl(\|v(s)\|_{H^{1}(B_{8}(z)\cap\Omega_{\lambda})}
+\|v_{s}(s)\|_{L^{2}(B_{8}(z)\cap\Omega_{\lambda})}\Bigr)=\infty.
\]
\end{lemma}

\begin{proof}
Suppose that the quantity in question were bounded by some $N'$ along a
sequence $s_{j}\uparrow S_{\max}$.  Applying Lemma \ref{lem:glue} with the data
$(v,v_{s})(s_{j})$, for an index $j$ such that $S_{\max}-s_{j}<S_{1}(N')$, and
using that \eqref{eq:eqD} does not depend on time explicitly, we obtain
a solution on $[s_{j},s_{j}+S_{1}(N')]$.  Concatenating it with $v|_{[0,s_{j}]}$
produces a function which lies in the class \eqref{eq:class} on
$[0,s_{j}+S_{1}(N')]$, because the two pieces share the value
$(v,v_{s})(s_{j})$ in $H_{0}^{1}\times L^{2}$, and which satisfies
\eqref{eq:weak} on that interval, since \eqref{eq:weak} is additive in
$[s_{1},s_{2}]$.  We would thus extend the solution beyond $S_{\max}$.  By
Lemma \ref{lem:unique} the extension coincides with $v$ where both are defined,
which contradicts the maximality of $S_{\max}$.
\end{proof}

Lemmas \ref{lem:data} and \ref{lem:glue} together are already enough to prove
the first lower bound of Theorem \ref{thm:lower}, and this will be done in
Section \ref{sec:proof}.  For the second lower bound, the size of the damping
coefficient has to be exploited, and this is the object of the last group of
auxiliary results.

\subsection{Growth of the local energy for a large damping coefficient}\label{sec:growth}

In this subsection we assume \eqref{eq:data}, \eqref{eq:exponents},
$\varrho\geq1$ and $q_{c}<q<p$, so that
\[
\lambda=\varrho^{(p-1)/2}\geq1,
\qquad
A=\varepsilon_{\varrho}=\varrho^{\kappa}\geq1,
\qquad
\kappa=\frac{q(p+1)-2p}{2}>0,
\]
and we denote by $v$ the maximal weak solution of
\eqref{eq:rescaled-pb}--\eqref{eq:rescaled-data}.  Such a maximal solution
exists: by Lemma \ref{lem:unique} any two weak solutions in the class
\eqref{eq:class} with the data \eqref{eq:rescaled-data} coincide on the
intersection of their intervals of existence, so the union of all of them is
again a weak solution, defined on an interval $[0,S_{\max})$; by construction
$S_{\max}$ is the maximal existence time $S^{*}_{\lambda}$ of Lemma
\ref{lem:rescale}.  We extend $v$ by zero to the whole of $\mathbb{R}^{n}$,
which is legitimate because $v(s)\in H_{0}^{1}(\Omega_{\lambda})$.  We fix once
and for all a function $\chi\in C_{c}^{\infty}(\mathbb{R}^{n})$ with
\begin{equation}\label{eq:cutoff}
0\leq\chi\leq1,
\qquad
\chi\equiv1\ \text{on }B_{4}(0),
\qquad
\operatorname{supp}\chi\subset B_{5}(0),
\qquad
|\nabla\chi|\leq2,
\end{equation}
and we set $\chi_{z}=\chi(\cdot-z)$.  The quantities to be estimated are the
uniformly local energy and the corresponding dissipation,
\begin{equation}\label{eq:EandD}
\mathcal{E}_{z}(s)=\frac{1}{2}\int_{\mathbb{R}^{n}}
\chi_{z}^{2}\bigl(v_{s}^{2}+|\nabla v|^{2}+v^{2}\bigr)dy,
\qquad
\mathcal{E}(s)=\sup_{z\in\mathbb{R}^{n}}\mathcal{E}_{z}(s),
\qquad
\mathcal{D}_{z}(s)=A\int_{\mathbb{R}^{n}}\chi_{z}^{2}|v_{s}|^{q+1}dy.
\end{equation}
Three elementary observations will be used repeatedly.  First, since
$\chi_{z'}\equiv1$ on $B_{4}(z')$, every ball $B_{r}(z)$ with $r\geq1$ is
covered by a number of balls $B_{4}(z_{i})$ depending only on $r$ and $n$, so
that
\begin{equation}\label{eq:cover}
\|v(s)\|_{H^{1}(B_{r}(z))}^{2}+\|v_{s}(s)\|_{L^{2}(B_{r}(z))}^{2}
\leq C(r,n)\,\mathcal{E}(s),
\qquad z\in\mathbb{R}^{n},\ r\geq1.
\end{equation}
Second, $\mathcal{E}$ is measurable: the functions $\chi_{z}$ being translates
of a single function and $v_{s}^{2}+|\nabla v|^{2}+v^{2}$ being integrable for
each fixed $s$, the continuity of translations in $L^{1}$ makes
$z\mapsto\mathcal{E}_{z}(s)$ continuous, so that
\begin{equation}\label{eq:meas}
\mathcal{E}(s)=\sup_{z\in\mathbb{Q}^{n}}\mathcal{E}_{z}(s),
\end{equation}
which exhibits $\mathcal{E}$ as a countable supremum of continuous functions of
$s$, hence as a lower semicontinuous function.  Third, $\mathcal{E}$ is finite
on $[0,S]$ for every $S<S_{\max}$, being bounded there by
$\|v(s)\|_{H^{1}(\Omega_{\lambda})}^{2}+\|v_{s}(s)\|_{L^{2}(\Omega_{\lambda})}^{2}$,
which is continuous in $s$.  We also record that the exponent produced by
Young's inequality in \eqref{eq:young-intro} is admissible for the Sobolev
embedding: since $q>1$ we have $(q+1)/q<2$, and $2p<2n/(n-2)$ for $n\geq3$ is
the assumption $p<n/(n-2)$ in \eqref{eq:exponents}, so that
\begin{equation}\label{eq:admissible}
\frac{p(q+1)}{q}<2p<
\begin{cases}
\dfrac{2n}{n-2}, & n\geq3,\\[2mm]
\infty, & n\leq2,
\end{cases}
\qquad\text{hence}\qquad
H^{1}(B)\hookrightarrow L^{p(q+1)/q}(B)
\end{equation}
on balls $B$, with a constant depending only on $n$, $p$, $q$ and the radius.

The next lemma is the key step for the second lower bound.  It states that the
uniformly local energy of the rescaled solution increases at the rate
$A^{-1/q}$, which is small when the damping coefficient is large.  Only
$A\geq1$ is used, and not the strict inequality $q>q_{c}$.

\begin{lemma}\label{lem:growth}
Assume \eqref{eq:exponents} and $A\geq1$.  There exists a constant
$C_{6}=C_{6}(n,p,q)$ such that, for every $z\in\mathbb{R}^{n}$ and every
$0\leq s<S_{\max}$,
\begin{equation}\label{eq:growth}
\mathcal{E}_{z}(s)+\frac{1}{4}\int_{0}^{s}\mathcal{D}_{z}(\tau)\,d\tau
\leq\mathcal{E}_{z}(0)
+C_{6}A^{-1/q}\int_{0}^{s}\bigl(1+\mathcal{E}(\tau)\bigr)^{\Gamma}d\tau,
\qquad
\Gamma=\frac{p(q+1)}{2q}>1.
\end{equation}
\end{lemma}

\begin{proof}
Let $S<S_{\max}$ and $s\leq S$.  The restriction of $v$ to $[0,S]$ is a weak
solution of \eqref{eq:rescaled-pb} in the class \eqref{eq:class}, so Lemma
\ref{lem:loc-energy} applies with $h=v|v|^{p-1}$, with $s_{1}=0$, $s_{2}=s$ and
with the function $\chi_{z}$ of \eqref{eq:cutoff}.  Adding to
\eqref{eq:loc-energy} the identity \eqref{eq:L2loc} with $\chi=\chi_{z}$, which
is legitimate because $v\in C^{1}([0,S];L^{2})$, we obtain
\begin{equation}\label{eq:loc-ineq}
\mathcal{E}_{z}(s)+\int_{0}^{s}\mathcal{D}_{z}
\leq\mathcal{E}_{z}(0)
+\int_{0}^{s}\bigl[L_{1}(\tau)+L_{2}(\tau)+L_{3}(\tau)\bigr]d\tau,
\end{equation}
where
\[
L_{1}=\int\chi_{z}^{2}|v|^{p}|v_{s}|,
\qquad
L_{2}=\int\chi_{z}^{2}|v||v_{s}|,
\qquad
L_{3}=2\int\chi_{z}|\nabla\chi_{z}||v_{s}||\nabla v|.
\]
Each of the three terms is estimated by means of the elementary inequality
\begin{equation}\label{eq:young}
ab\leq\frac{1}{4}a^{q+1}+C(q)\,b^{(q+1)/q},
\qquad a,b\geq0,
\end{equation}
in which the first factor carries the dissipation.

Consider $L_{1}$.  Writing
$\chi_{z}^{2}|v|^{p}|v_{s}|
=\bigl[\chi_{z}^{2/(q+1)}|v_{s}|\bigr]\cdot
\bigl[\chi_{z}^{2q/(q+1)}|v|^{p}\bigr]$
and applying H\"older's inequality with the exponents $q+1$ and $(q+1)/q$, we
obtain
\[
L_{1}\leq A^{-\frac{1}{q+1}}\mathcal{D}_{z}^{\frac{1}{q+1}}
\Bigl(\int\chi_{z}^{2}|v|^{\frac{p(q+1)}{q}}\Bigr)^{\frac{q}{q+1}},
\]
so that \eqref{eq:young} gives
\[
L_{1}\leq\frac{1}{4}\mathcal{D}_{z}
+C(q)A^{-1/q}\int\chi_{z}^{2}|v|^{\frac{p(q+1)}{q}}.
\]
By \eqref{eq:admissible}, by $\operatorname{supp}\chi_{z}\subset B_{5}(z)$ and
by \eqref{eq:cover},
\[
\int\chi_{z}^{2}|v|^{\frac{p(q+1)}{q}}
\leq\int_{B_{5}(z)}|v|^{\frac{p(q+1)}{q}}
\leq C\|v\|_{H^{1}(B_{5}(z))}^{\frac{p(q+1)}{q}}
\leq C\mathcal{E}^{\Gamma},
\]
with $\Gamma=p(q+1)/(2q)$.

Consider $L_{2}$.  The same computation with $|v|$ in place of $|v|^{p}$ gives
\[
L_{2}\leq\frac{1}{4}\mathcal{D}_{z}
+C(q)A^{-1/q}\int_{B_{5}(z)}|v|^{\frac{q+1}{q}},
\]
and, since $(q+1)/q<2$, H\"older's inequality on the ball together with
\eqref{eq:cover} gives
$\int_{B_{5}(z)}|v|^{(q+1)/q}
\leq C\bigl(\int_{B_{5}(z)}v^{2}\bigr)^{\frac{q+1}{2q}}
\leq C\mathcal{E}^{\frac{q+1}{2q}}$.

Consider $L_{3}$.  Writing
$\chi_{z}|\nabla\chi_{z}||v_{s}||\nabla v|
=\bigl[\chi_{z}^{2/(q+1)}|v_{s}|\bigr]\cdot
\bigl[\chi_{z}^{(q-1)/(q+1)}|\nabla\chi_{z}||\nabla v|\bigr]$
and applying H\"older's inequality with the same exponents, together with
$\chi_{z}\leq1$ and $|\nabla\chi_{z}|\leq2$, we obtain
\[
L_{3}\leq2A^{-\frac{1}{q+1}}\mathcal{D}_{z}^{\frac{1}{q+1}}
\Bigl(C\int_{B_{5}(z)}|\nabla v|^{\frac{q+1}{q}}\Bigr)^{\frac{q}{q+1}}.
\]
Since $(q+1)/q<2$, H\"older's inequality on $B_{5}(z)$ gives
\[
\Bigl(\int_{B_{5}(z)}|\nabla v|^{\frac{q+1}{q}}\Bigr)^{\frac{q}{q+1}}
\leq C\Bigl(\int_{B_{5}(z)}|\nabla v|^{2}\Bigr)^{1/2}
\leq C\mathcal{E}^{1/2},
\]
whence, by \eqref{eq:young},
\[
L_{3}\leq\frac{1}{4}\mathcal{D}_{z}
+C(q)A^{-1/q}\mathcal{E}^{\frac{q+1}{2q}}.
\]

Inserting the three estimates into \eqref{eq:loc-ineq} and absorbing
$\frac{3}{4}\int_{0}^{s}\mathcal{D}_{z}$ into the left-hand side, which is
legitimate because $v_{s}\in L^{q+1}$ makes that integral finite, we obtain
\eqref{eq:growth}, since both $\mathcal{E}^{\Gamma}$ and
$\mathcal{E}^{(q+1)/(2q)}$ are bounded by $(1+\mathcal{E})^{\Gamma}$, the
inequality $\Gamma\geq(q+1)/(2q)$ holding because $p>1$.  Finally $\Gamma>1$:
we have $2q/(q+1)<q<p$ because $q>1$, whence $p(q+1)>2q$.
\end{proof}

The gain in \eqref{eq:growth} is converted into a lower bound for the
existence time by a comparison argument of Bihari type.

\begin{lemma}\label{lem:bihari}
Assume \eqref{eq:exponents} and $A\geq1$, and let $K\geq\mathcal{E}(0)$.  With
$C_{6}$ and $\Gamma$ as in Lemma \ref{lem:growth}, set
\begin{equation}\label{eq:c0}
c_{K}=\frac{(1+K)^{1-\Gamma}}{2(\Gamma-1)C_{6}}.
\end{equation}
Then
\begin{equation}\label{eq:bihari}
S_{\max}>c_{K}A^{1/q},
\qquad\text{and}\qquad
\mathcal{E}(s)\leq2^{\frac{1}{\Gamma-1}}(1+K)
\quad\text{for }0\leq s\leq c_{K}A^{1/q}.
\end{equation}
\end{lemma}

Any $K\geq\mathcal{E}(0)$ is allowed here, and not only
$K=\mathcal{E}(0)$, because $\mathcal{E}(0)$ depends on $\lambda$ whereas the
upper bound for it provided by Lemma \ref{lem:data} does not; since
$1-\Gamma<0$, the constant $c_{K}$ is nonincreasing in $K$, and it will be
evaluated at that uniform bound.

\begin{proof}
Let $S<S_{\max}$ and put
\[
Y(s)=K+C_{6}A^{-1/q}
\int_{0}^{s}\bigl(1+\mathcal{E}(\tau)\bigr)^{\Gamma}d\tau,
\qquad 0\leq s\leq S.
\]
The integrand is measurable by \eqref{eq:meas} and bounded on $[0,S]$, so $Y$
is absolutely continuous and nondecreasing, and taking the supremum over $z$
in \eqref{eq:growth}, together with $\mathcal{E}_{z}(0)\leq\mathcal{E}(0)\leq K$,
gives $\mathcal{E}\leq Y$ on $[0,S]$.  Consequently, for almost every $s$,
\[
Y'(s)=C_{6}A^{-1/q}\bigl(1+\mathcal{E}(s)\bigr)^{\Gamma}
\leq C_{6}A^{-1/q}\bigl(1+Y(s)\bigr)^{\Gamma},
\]
and, since $\Gamma>1$,
\[
-\frac{1}{\Gamma-1}\frac{d}{ds}\bigl(1+Y\bigr)^{1-\Gamma}
\leq C_{6}A^{-1/q},
\qquad\text{whence}\qquad
\bigl(1+Y(s)\bigr)^{1-\Gamma}
\geq(1+K)^{1-\Gamma}-(\Gamma-1)C_{6}A^{-1/q}s.
\]
If $s\leq c_{K}A^{1/q}$, with $c_{K}$ as in \eqref{eq:c0}, the right-hand side
is at least $\frac{1}{2}(1+K)^{1-\Gamma}$, and therefore
\[
1+Y(s)\leq2^{\frac{1}{\Gamma-1}}(1+K)=:K_{*}.
\]
Hence $\mathcal{E}\leq K_{*}$ on $[0,\min\{S,c_{K}A^{1/q}\}]$ for every
$S<S_{\max}$, that is, on $[0,\min\{S_{\max},c_{K}A^{1/q}\})$.

Assume now that $S_{\max}\leq c_{K}A^{1/q}$.  Then $\mathcal{E}\leq K_{*}$ on
$[0,S_{\max})$, and \eqref{eq:cover} with $r=8$ gives
\[
\sup_{z\in\mathbb{R}^{n}}
\Bigl(\|v(s)\|_{H^{1}(B_{8}(z)\cap\Omega_{\lambda})}
+\|v_{s}(s)\|_{L^{2}(B_{8}(z)\cap\Omega_{\lambda})}\Bigr)
\leq C(n)K_{*}^{1/2}
\qquad\text{for all }s<S_{\max},
\]
which contradicts Lemma \ref{lem:blowupalt}, applied on $\Omega_{\lambda}$ with
$\varepsilon=A$.  Therefore $S_{\max}>c_{K}A^{1/q}$, and \eqref{eq:bihari}
follows.
\end{proof}

With Lemmas \ref{lem:rescale}, \ref{lem:data}, \ref{lem:glue} and
\ref{lem:bihari} established, the proofs of the main results reduce to a
computation of exponents, which is carried out in the next section.

\section{Proof of the main results}\label{sec:proof}

We begin with the first lower bound of Theorem \ref{thm:lower}, which uses
only Lemmas \ref{lem:data} and \ref{lem:glue}.

\begin{proof}[Proof of Theorem \ref{thm:lower}, part $(i)$]
Let $\varrho\geq1$, so that $\lambda=\varrho^{(p-1)/2}\geq1$.  By Lemma
\ref{lem:data} applied with $R=8$, the rescaled data \eqref{eq:rescaled-data}
satisfy the hypothesis \eqref{eq:unifloc-hyp} of Lemma \ref{lem:glue} with
$N=K_{8}$, a constant which does not depend on $\lambda$ and hence not on
$\varrho$.  Lemma \ref{lem:glue}, applied with $\varepsilon=\varepsilon_{\varrho}$,
then provides a weak solution of \eqref{eq:rescaled-pb} on the interval
$[0,S_{1}]$, where $S_{1}=S_{1}(K_{8};n,p)>0$ depends neither on $\lambda$ nor
on $\varepsilon_{\varrho}$.  In particular the conclusion is the same whether
$\varepsilon_{\varrho}$ tends to zero, is equal to one, or tends to infinity,
that is, in each of the three regimes described in Section \ref{sec:intro}.
Since this solution is unique by Lemma \ref{lem:unique}, we have
$S^{*}_{\lambda}\geq S_{1}$, and therefore, by \eqref{eq:times},
\[
T^{*}(\varrho)\geq\frac{S_{1}}{\lambda}=S_{1}\varrho^{-(p-1)/2},
\]
which is \eqref{eq:lower1} with $C_{l}=S_{1}$.  Finally, $C_{l}$ has the
dependence claimed in Theorem \ref{thm:lower}: the damping term enters Lemmas
\ref{lem:localexist} and \ref{lem:glue} only through its sign, so that $S_{1}$
involves neither $q$ nor $\varepsilon_{\varrho}$, and $K_{8}$ depends only on
$n$ and the norms in \eqref{eq:data}, by Lemma \ref{lem:data}, so that $S_{1}$
involves $\Omega$ only through the values of these norms.
\end{proof}

In view of \eqref{eq:upper}, the exponent in \eqref{eq:lower1} is optimal when
$q\leq q_{c}$.  For $q>q_{c}$ the upper bound in \eqref{eq:upper} is of the
smaller order $\varrho^{-(p-q)/q}$, and the size of the damping term has to be
used; this is what Lemma \ref{lem:bihari} does.

\begin{proof}[Proof of Theorem \ref{thm:lower}, part $(ii)$]
Let $\varrho\geq1$; then $\lambda\geq1$ and, since $\kappa>0$,
$A=\varrho^{\kappa}\geq1$.  Since $\operatorname{supp}\chi_{z}\subset B_{5}(z)$
and $\chi_{z}\leq1$, Lemma \ref{lem:data} with $R=5$ gives
\[
\mathcal{E}(0)\leq\frac{1}{2}\sup_{z\in\mathbb{R}^{n}}
\Bigl(\|f_{\lambda}\|_{H^{1}(B_{5}(z)\cap\Omega_{\lambda})}^{2}
+\|\lambda^{-1}g_{\lambda}\|_{L^{2}(B_{5}(z)\cap\Omega_{\lambda})}^{2}\Bigr)
\leq K_{5}^{2},
\]
a bound which does not depend on $\lambda$, hence not on $\varrho$.  We apply
Lemma \ref{lem:bihari} with $K=K_{5}^{2}$: the constant
\[
c_{*}:=c_{K_{5}^{2}}=\frac{(1+K_{5}^{2})^{1-\Gamma}}{2(\Gamma-1)C_{6}}
\]
then depends only on $n$, $p$, $q$ and the norms appearing in \eqref{eq:data},
since $C_{6}$ and $\Gamma$ depend only on $n$, $p$ and $q$, and Lemma
\ref{lem:bihari} gives $S^{*}_{\lambda}=S_{\max}>c_{*}A^{1/q}$.  By
\eqref{eq:times},
\[
T^{*}(\varrho)\geq\lambda^{-1}c_{*}A^{1/q}
=c_{*}\,\varrho^{-\frac{p-1}{2}+\frac{\kappa}{q}}.
\]
It remains to compute the exponent.  By \eqref{eq:params},
\[
-\frac{p-1}{2}+\frac{\kappa}{q}
=\frac{-q(p-1)+q(p+1)-2p}{2q}
=\frac{2q-2p}{2q}
=-\frac{p-q}{q},
\]
which gives \eqref{eq:lower2} with $C_{l}'=c_{*}$.
\end{proof}

We now combine the lower bounds with the upper bounds \eqref{eq:upper}.

\begin{proof}[Proof of Theorem \ref{thm:two-sided}]
Since $f\not\equiv0$, Proposition \ref{prop:upper} provides $\varrho_{0}\geq1$
such that, for every $\varrho\geq\varrho_{0}$, the solution of \eqref{eq:main}
blows up in finite time and satisfies \eqref{eq:upper} with constants $C_{u}$
and $C_{u}'$ independent of $\varrho$.  The lower bounds \eqref{eq:lower1} and
\eqref{eq:lower2} hold for every $\varrho\geq1$ by Theorem \ref{thm:lower}.
For $q\leq q_{c}$, the lower bound in \eqref{eq:two-sided-sub} is
\eqref{eq:lower1} and the upper bound is the first estimate in
\eqref{eq:upper}; for $q>q_{c}$, the lower bound in \eqref{eq:two-sided-over}
is \eqref{eq:lower2} and the upper bound is the second estimate in
\eqref{eq:upper}.  The identification of $\vartheta(p,q)$ in
\eqref{eq:vartheta} was observed after Theorem \ref{thm:lower}, and
\eqref{eq:two-sided} follows with $C=C_{l}$, $C'=C_{u}$ when
$1<q\leq q_{c}$, and with $C=C_{l}'$, $C'=C_{u}'$ when $q_{c}<q<p$.  This
completes the proof of Theorem \ref{thm:two-sided}.
\end{proof}

\begin{remarks}\label{rem:proof}
\begin{enumerate}
\item[$(i)$] The two parts of Theorem \ref{thm:lower} use the damping term in
different ways.  In part $(i)$ only its sign is used, through the monotonicity
of $\zeta\mapsto\zeta|\zeta|^{q-1}$.  In part $(ii)$ its size is used as well:
the dissipation $A\int\chi^{2}|v_{s}|^{q+1}$ forces the velocity towards the
level $|v_{s}|\sim A^{-1/q}|v|^{p/q}$ at which the two terms of the equation
balance, so that the rate at which the source can feed the local energy is
itself of order $A^{-1/q}$.
\item[$(ii)$] The time scale $A^{1/q}$ extracted from Lemmas \ref{lem:growth}
and \ref{lem:bihari} is optimal for the lifespan under the blow-up hypotheses
\eqref{eq:hyp-blowup}: the upper bound in \eqref{eq:upper} shows that, up to
multiplicative constants, the solution cannot in general be continued beyond
this scale.  For $q<q_{c}$ one has
$A=\varepsilon_{\varrho}\leq1$, and the argument produces nothing beyond part
$(i)$, in agreement with the first estimate of \eqref{eq:upper}.
\item[$(iii)$] The constant $C_{l}'=c_{*}$ of \eqref{eq:lower2} does not
degenerate as $q$ decreases to $q_{c}$.  The only factor in \eqref{eq:c0} which
could do so is $(\Gamma-1)^{-1}$, and it does not: the exponent $\Gamma=p(q+1)/(2q)$ is a decreasing function of
$q$, so that $\Gamma-1\geq\frac{p+1}{2}-1=\frac{p-1}{2}>0$ on the whole range
$q_{c}<q<p$, and at the threshold itself one has
$q_{c}+1=\frac{3p+1}{p+1}$ and hence $\Gamma=\frac{3p+1}{4}$, that is,
$\Gamma-1=\frac{3(p-1)}{4}$.  The second lower-bound scale therefore extends
continuously to the first one as $q\downarrow q_{c}$, instead of being lost
there, in agreement with the fact that the two branches of $\vartheta(p,q)$ in
\eqref{eq:vartheta} meet at that value.
\item[$(iv)$] No smoothness of $f$ and $g$ beyond \eqref{eq:data} is used, and
no restriction on the dimension beyond \eqref{eq:exponents}, since by
\eqref{eq:admissible} the exponent produced by Young's inequality is admissible
in the whole range \eqref{eq:exponents}.  All the estimates take place in the
energy class \eqref{eq:class} itself, which is also the class in which the
upper bounds \eqref{eq:upper} are proved in Appendix \ref{sec:appendix}.
\end{enumerate}
\end{remarks}

\section{Concluding remarks and open problems}\label{sec:open}

The lifespan of the solution of \eqref{eq:main} at large amplitude is now
known up to multiplicative constants:
\[
T^{*}(\varrho)\asymp\varrho^{-\vartheta(p,q)},
\qquad
\vartheta(p,q)=\min\Bigl\{\frac{p-1}{2},\,\frac{p-q}{q}\Bigr\},
\]
for nontrivial $f$ and all sufficiently large $\varrho$.  The upper bounds come
from the concavity method of \cite{BHK}; the lower bounds come from rescaling
the problem before applying any existence theory, and from two properties of
the damping term which are used separately, its monotonicity in the first
regime and its size in the second.  The threshold $q_{c}=2p/(p+1)$ between the
two regimes is the exponent at which the damping term is invariant under the
rescaling \eqref{eq:scaling}, and the gap left by the energy argument of
\cite{BHK}, whose exponent was $1-p$, is closed.

Three questions are left open.  The first is the determination of the sharp
multiplicative constants, that is, whether
$\varrho^{\vartheta(p,q)}T^{*}(\varrho)$ converges as $\varrho\to\infty$ and,
if so, the identification of its limit in terms of the profiles $f$ and $g$;
the constants produced here and in \cite{BHK} are far apart, and neither
argument is designed to be sharp.  The second is the blow-up rate: the
heuristic of Remarks \ref{rem:main} $(ii)$ suggests the formal amplitude
exponent $1/\vartheta(p,q)$, changing branch at $q_{c}$, but establishing a
corresponding pointwise or norm-specific blow-up rate for \eqref{eq:main}
remains open.  The third is
the extent to which the method survives a loss of scaling invariance, for
instance for the variable-exponent equations of \cite{Messaoudi2017}, where $p$
and $q$ depend on $x$ and the rescaling \eqref{eq:scaling} no longer maps the
problem to one of the same form; the two mechanisms used above, the
monotonicity of the damping term and the quantitative rate $A^{-1/q}$, are
local in nature and may have counterparts in that setting.

The two halves of Theorem \ref{thm:two-sided} use the boundedness of $\Omega$
in different ways.  In the lower bounds it enters qualitatively only: through the finite cover in the
proof of Lemma \ref{lem:glue}, which a countable cover with bounded overlap
would replace, and through the setting in which the existence theory of Lemma
\ref{lem:aux} is quoted.  No constant in Theorem \ref{thm:lower} depends on
$|\Omega|$.  In the upper bounds it enters quantitatively, through the constant
$C_{7}=|\Omega|^{(p-q)/(p+1)}$ and through $c_{\Omega}$ in the proof of
Proposition \ref{prop:upper}, both of which are infinite when $|\Omega|$ is.
We therefore expect the two lower bounds to persist on unbounded domains, and
in particular on $\Omega=\mathbb{R}^{n}$ for profiles satisfying
\eqref{eq:data}, but we have not carried out the details; whether
$\vartheta(p,q)$ still describes the lifespan there is a further open question,
and one which the concavity method in its present form does not answer.

\appendix
\section{The upper bounds}\label{sec:appendix}

The upper bounds \eqref{eq:upper} were obtained in \cite{BHK} by the concavity
method.  Theorem \ref{thm:two-sided} requires them in a form in which the
independence of the constants of $\varrho$ is verified at every step, and this
is not automatic: the auxiliary parameters of the concavity argument are
constrained by quantities such as $\beta_{0}$ and $H(0)$, which themselves grow
with $\varrho$.  We therefore go through the argument of \cite{BHK} once,
keeping every constant explicit; the new elements are the choice of the
parameters in \eqref{eq:Lambda-choice}--\eqref{eq:rho0}, which decouples them
from $\varrho$, and the lower bound for $\gamma(0)$ that this choice supplies.
Throughout, $u$ is the weak solution of \eqref{eq:main} in the class
\eqref{eq:class}, with $D=\Omega$ and $\varepsilon=1$, on $[0,T^{*})$, the time
variable is $t$, and $E$ is the energy \eqref{eq:energy}.

Two identities are used.  First, $E$ is absolutely continuous on compact
subintervals of $[0,T^{*})$ and
\begin{equation}\label{eq:E-AC}
E'(t)=-\|u_{t}(t)\|_{q+1}^{q+1}
\qquad\text{for almost every }t\in[0,T^{*}):
\end{equation}
this is Lemma \ref{lem:loc-energy} with $h=u|u|^{p-1}$ and $\chi\equiv1$ on a
ball containing $\overline{\Omega}$, once one observes that
$t\mapsto\|u(t)\|_{p+1}^{p+1}$ is of class $C^{1}$ with derivative
$(p+1)\int_{\Omega}|u|^{p-1}uu_{t}\,dx$, because $v\mapsto|v|^{p-1}v$ is
continuous from $L^{2p}(\Omega)$ into $L^{2}(\Omega)$,
$u\in C([0,T];L^{2p}(\Omega))$ by \eqref{eq:exponents} and
$u\in C^{1}([0,T];L^{2}(\Omega))$.  Second,
$\Phi(t)=\int_{\Omega}u(t)u_{t}(t)\,dx$ is absolutely continuous on compact
subintervals of $[0,T^{*})$ and
\begin{equation}\label{eq:Phi-AC}
\Phi'(t)=\|u_{t}(t)\|_{2}^{2}-\|\nabla u(t)\|_{2}^{2}+\|u(t)\|_{p+1}^{p+1}
-\int_{\Omega}u\,u_{t}|u_{t}|^{q-1}\,dx
\qquad\text{for almost every }t,
\end{equation}
which is \eqref{eq:weak} with the admissible test function $\varphi=u$.  No
second derivative in time is needed.

Let
\begin{equation}\label{eq:w-def}
\mathcal{W}(\zeta)=\frac{\zeta}{2}
-\frac{B_{*}^{p+1}}{p+1}\,\zeta^{\frac{p+1}{2}},
\qquad\zeta\geq0 .
\end{equation}
Since $\mathcal{W}'$ vanishes exactly at $\zeta=\mu_{1}$, the function
$\mathcal{W}$ increases on $(0,\mu_{1})$ and decreases on $(\mu_{1},\infty)$ to
$-\infty$; moreover $\mathcal{W}(\mu_{1})=E_{1}$, because
$B_{*}^{p+1}\mu_{1}^{(p+1)/2}=\mu_{1}$ by \eqref{eq:E1}, and
$E(t)\geq\mathcal{W}(\|\nabla u(t)\|_{2}^{2})$ by the Sobolev inequality.  The
next lemma is the potential-well lemma of Vitillaro \cite{Vitillaro2000}, see
also \cite{SunRenGao2016,BHK}, completed by the lower bound \eqref{eq:beta0},
which is what will absorb the constant $E_{1}$ below.

\begin{lemma}\label{lem:well}
Assume \eqref{eq:data}, \eqref{eq:exponents} and \eqref{eq:hyp-blowup}, and
let $\mu_{2}>\mu_{1}$ be the solution of $\mathcal{W}(\mu_{2})=E(0)$.  Then,
for every $t\in[0,T^{*})$,
\begin{equation}\label{eq:well}
\|\nabla u(t)\|_{2}^{2}\geq\mu_{2},
\qquad
0<H(0)\leq H(t)\leq\frac{1}{p+1}\|u(t)\|_{p+1}^{p+1},
\qquad
\|u(t)\|_{p+1}^{p+1}\geq\beta_{0},
\end{equation}
where $H(t)=E_{1}-E(t)$ and
\begin{equation}\label{eq:beta0}
\beta_{0}=(p+1)\Bigl(\frac{\mu_{2}}{2}-E(0)\Bigr)
=B_{*}^{p+1}\mu_{2}^{\frac{p+1}{2}}>\mu_{1}.
\end{equation}
\end{lemma}

\begin{proof}
Since $\mathcal{W}(\mu_{1})=E_{1}>E(0)$, the number $\mu_{2}$ exists and is
unique.  At $t=0$, $\mathcal{W}(\|\nabla u_{0}\|_{2}^{2})\leq
E(0)=\mathcal{W}(\mu_{2})$ with $\|\nabla u_{0}\|_{2}^{2}>\mu_{1}$ by
\eqref{eq:hyp-blowup}, hence $\|\nabla u_{0}\|_{2}^{2}\geq\mu_{2}$.  If
$\|\nabla u(t_{0})\|_{2}^{2}<\mu_{2}$ for some $t_{0}$, continuity gives
$t_{1}\in(0,t_{0})$ with $\mu_{1}<\|\nabla u(t_{1})\|_{2}^{2}<\mu_{2}$,
and then $E(t_{1})\geq\mathcal{W}(\|\nabla u(t_{1})\|_{2}^{2})
>\mathcal{W}(\mu_{2})=E(0)$, which contradicts $E(t_{1})\leq E(0)$, a
consequence of \eqref{eq:E-AC}; this proves the first inequality.  The second
follows from \eqref{eq:E-AC}, \eqref{eq:hyp-blowup} and
$E_{1}<\frac{1}{2}\mu_{1}<\frac{1}{2}\mu_{2}$, since
$H=E_{1}-\frac{1}{2}\|u_{t}\|_{2}^{2}-\frac{1}{2}\|\nabla u\|_{2}^{2}
+\frac{1}{p+1}\|u\|_{p+1}^{p+1}$.  The third follows from
$\frac{1}{p+1}\|u\|_{p+1}^{p+1}
=\frac{1}{2}\|u_{t}\|_{2}^{2}+\frac{1}{2}\|\nabla u\|_{2}^{2}-E(t)
\geq\frac{1}{2}\mu_{2}-E(0)$.  The two expressions of $\beta_{0}$ agree
because $E(0)=\mathcal{W}(\mu_{2})$, and
$\beta_{0}>B_{*}^{p+1}\mu_{1}^{(p+1)/2}=\mu_{1}$ because $\mu_{2}>\mu_{1}$.
\end{proof}

\begin{proposition}\label{prop:upper}
Assume \eqref{eq:data}, \eqref{eq:exponents} and $f\not\equiv0$, and let
\begin{equation}\label{eq:beta}
\beta=\min\Bigl\{\frac{p-q}{q(p+1)},\ \frac{p-1}{2(p+1)}\Bigr\}
=\begin{cases}
\dfrac{p-1}{2(p+1)}, & q\leq q_{c},\\[3mm]
\dfrac{p-q}{q(p+1)}, & q>q_{c},
\end{cases}
\qquad\text{so that}\qquad
(p+1)\beta=\vartheta(p,q).
\end{equation}
There exist $\varrho_{0}\geq1$ and $C>0$, depending only on $p$, $q$,
$|\Omega|$, $B_{*}$, $\|\nabla f\|_{2}$, $\|f\|_{2}$, $\|f\|_{p+1}$ and
$\|g\|_{2}$, such that for every $\varrho\geq\varrho_{0}$ the solution of
\eqref{eq:main} blows up in finite time and
\begin{equation}\label{eq:upper-app}
T^{*}(\varrho)\leq C\varrho^{-\vartheta(p,q)},
\end{equation}
which is \eqref{eq:upper}.
\end{proposition}

\begin{proof}
\emph{Parameters.}  Put $C_{7}=|\Omega|^{\frac{p-q}{p+1}}$, fix $\Lambda>1$ with
\begin{equation}\label{eq:Lambda-choice}
C_{7}\Lambda^{-(q+1)}\leq\frac{p-1}{4(p+1)},
\qquad\text{and set}\qquad
\theta=(1-\beta)\,\Lambda^{-\frac{q+1}{q}};
\end{equation}
$\Lambda$ and $\theta$ depend only on $p$, $q$ and $|\Omega|$.  Since
$f\not\equiv0$, \eqref{eq:E0} gives $E(0)\to-\infty$ and
$H(0)=E_{1}-E(0)\geq\frac{\varrho^{p+1}}{p+1}\|f\|_{p+1}^{p+1}-C\varrho^{2}
\to\infty$ as $\varrho\to\infty$, while
$(p+1)(1-\beta)\geq\frac{p+3}{2}>2$ by \eqref{eq:beta}; hence there is
$\varrho_{0}\geq1$ such that, for every $\varrho\geq\varrho_{0}$,
\begin{equation}\label{eq:rho0}
\begin{gathered}
\varrho^{2}\|\nabla f\|_{2}^{2}>\mu_{1},
\qquad
E(0)\leq-4\mu_{1},
\qquad
H(0)\geq\max\Bigl\{1,\ \frac{\varrho^{p+1}}{2(p+1)}\|f\|_{p+1}^{p+1}\Bigr\},\\[2pt]
\theta\varrho^{2}\|f\|_{2}\|g\|_{2}\leq\frac{1}{2}H(0)^{1-\beta}.
\end{gathered}
\end{equation}
The first two conditions imply \eqref{eq:hyp-blowup}, so Lemma
\ref{lem:well} applies, and
$\beta_{0}\geq-(p+1)E(0)\geq4(p+1)\mu_{1}$, whence
\begin{equation}\label{eq:beta0-large}
\frac{(p-1)\mu_{1}}{(p+1)\beta_{0}}\leq\frac{p-1}{4(p+1)}.
\end{equation}
Fix $\varrho\geq\varrho_{0}$ and set $\gamma=H^{1-\beta}+\theta\Phi$ on
$[0,T^{*})$.  Since $H\geq H(0)>0$, $\gamma$ is absolutely continuous on
compact subintervals, and by \eqref{eq:E-AC} and \eqref{eq:Phi-AC}, for almost
every $t$,
\begin{equation}\label{eq:gamma-prime}
\gamma'=(1-\beta)H^{-\beta}\|u_{t}\|_{q+1}^{q+1}
+\theta\Bigl(\|u_{t}\|_{2}^{2}-\|\nabla u\|_{2}^{2}+\|u\|_{p+1}^{p+1}
-\int_{\Omega}u\,u_{t}|u_{t}|^{q-1}\,dx\Bigr).
\end{equation}

\emph{Lower bound for $\gamma'$.}  Young's inequality with the exponents
$\frac{q+1}{q}$ and $q+1$, applied to
$\bigl(\Lambda H^{-\frac{\beta q}{q+1}}|u_{t}|^{q}\bigr)
\bigl(\Lambda^{-1}H^{\frac{\beta q}{q+1}}|u|\bigr)$, gives
\begin{equation}\label{eq:young-app}
\int_{\Omega}|u||u_{t}|^{q}\,dx
\leq\Lambda^{\frac{q+1}{q}}H^{-\beta}\|u_{t}\|_{q+1}^{q+1}
+\Lambda^{-(q+1)}H^{\beta q}\int_{\Omega}|u|^{q+1}\,dx,
\end{equation}
and H\"older's inequality together with \eqref{eq:well} gives
\begin{equation}\label{eq:G-est}
H^{\beta q}\int_{\Omega}|u|^{q+1}\,dx
\leq C_{7}H^{\beta q}\bigl(\|u\|_{p+1}^{p+1}\bigr)^{\frac{q+1}{p+1}}
\leq C_{7}(p+1)^{-\beta q}
\bigl(\|u\|_{p+1}^{p+1}\bigr)^{\frac{q+1}{p+1}+\beta q}
\leq C_{7}\|u\|_{p+1}^{p+1}.
\end{equation}
Here the second inequality uses $H\leq\frac{1}{p+1}\|u\|_{p+1}^{p+1}$ from
\eqref{eq:well}, and the third uses $(p+1)^{-\beta q}\leq1$ together with
$\frac{q+1}{p+1}+\beta q\leq1$, which is \eqref{eq:beta} rewritten as
$\beta q\leq\frac{p-q}{p+1}$, and with
$\|u\|_{p+1}^{p+1}\geq(p+1)H(0)\geq1$ by \eqref{eq:well} and \eqref{eq:rho0}.
We insert $-\int_{\Omega}u\,u_{t}|u_{t}|^{q-1}\geq-\int_{\Omega}|u||u_{t}|^{q}$,
\eqref{eq:young-app}, \eqref{eq:G-est} and the identity
$\|\nabla u\|_{2}^{2}=\frac{2}{p+1}\|u\|_{p+1}^{p+1}-\|u_{t}\|_{2}^{2}-2H+2E_{1}$
into \eqref{eq:gamma-prime}, and bound the constant term by
$2E_{1}=\frac{(p-1)\mu_{1}}{p+1}
\leq\frac{(p-1)\mu_{1}}{(p+1)\beta_{0}}\|u\|_{p+1}^{p+1}
\leq\frac{p-1}{4(p+1)}\|u\|_{p+1}^{p+1}$, by \eqref{eq:well} and
\eqref{eq:beta0-large}.  This yields
\[
\gamma'\geq\bigl[1-\beta-\theta\Lambda^{\frac{q+1}{q}}\bigr]
H^{-\beta}\|u_{t}\|_{q+1}^{q+1}
+2\theta\|u_{t}\|_{2}^{2}
+\theta\Bigl(\frac{p-1}{p+1}-\frac{C_{7}}{\Lambda^{q+1}}
-\frac{p-1}{4(p+1)}\Bigr)\|u\|_{p+1}^{p+1}
+2\theta H .
\]
The first bracket vanishes and the coefficient of
$\theta\|u\|_{p+1}^{p+1}$ is at least $\frac{p-1}{2(p+1)}$, both by
\eqref{eq:Lambda-choice}; since $\frac{p-1}{2(p+1)}<2$,
\begin{equation}\label{eq:gamma-lower}
\gamma'\geq\alpha_{1}\bigl(\|u_{t}\|_{2}^{2}+\|u\|_{p+1}^{p+1}+H\bigr),
\qquad
\alpha_{1}=\theta\,\frac{p-1}{2(p+1)}.
\end{equation}

\emph{Upper bound for $\gamma^{1/(1-\beta)}$.}  By the last condition in
\eqref{eq:rho0},
$\gamma(0)\geq H(0)^{1-\beta}-\theta\varrho^{2}\|f\|_{2}\|g\|_{2}
\geq\frac{1}{2}H(0)^{1-\beta}>0$, so $\gamma\geq\gamma(0)>0$ on $[0,T^{*})$
by \eqref{eq:gamma-lower}.  From $\gamma\leq H^{1-\beta}+\theta|\Phi|$,
the inequality $(a+b)^{r}\leq2^{r}(a^{r}+b^{r})$ for $a,b\geq0$, and
$|\Phi|\leq\|u\|_{2}\|u_{t}\|_{2}
\leq|\Omega|^{\frac{p-1}{2(p+1)}}\|u\|_{p+1}\|u_{t}\|_{2}$,
\[
\gamma^{\frac{1}{1-\beta}}
\leq2^{\frac{1}{1-\beta}}
\Bigl(H+\theta^{\frac{1}{1-\beta}}|\Phi|^{\frac{1}{1-\beta}}\Bigr),
\qquad
|\Phi|^{\frac{1}{1-\beta}}
\leq c_{\Omega}\Bigl(\|u_{t}\|_{2}^{2}
+\|u\|_{p+1}^{\frac{2}{1-2\beta}}\Bigr),
\]
with $c_{\Omega}=|\Omega|^{\frac{p-1}{2(p+1)(1-\beta)}}$, the second inequality
being Young's inequality with the conjugate exponents $2(1-\beta)$ and
$\frac{2(1-\beta)}{1-2\beta}$, both larger than one since $\beta<\frac{1}{2}$.
Moreover $\|u\|_{p+1}^{\frac{2}{1-2\beta}}
=\bigl(\|u\|_{p+1}^{p+1}\bigr)^{\frac{2}{(1-2\beta)(p+1)}}
\leq\|u\|_{p+1}^{p+1}$: indeed $(1-2\beta)(p+1)\geq2$, since
$(1-2\beta)(p+1)=2$ when $\beta=\frac{p-1}{2(p+1)}$ and
$(1-2\beta)(p+1)=\frac{q(p+1)-2(p-q)}{q}\geq2$ when
$\beta=\frac{p-q}{q(p+1)}$, the latter being equivalent to $q(p+1)\geq2p$,
which holds because then $q\geq q_{c}$; and $\|u\|_{p+1}^{p+1}\geq1$.
Therefore
\begin{equation}\label{eq:gamma-upper}
\gamma^{\frac{1}{1-\beta}}
\leq\alpha_{2}\bigl(\|u_{t}\|_{2}^{2}+\|u\|_{p+1}^{p+1}+H\bigr),
\qquad
\alpha_{2}=2^{\frac{1}{1-\beta}}
\bigl(1+\theta^{\frac{1}{1-\beta}}c_{\Omega}\bigr).
\end{equation}

\emph{Integration.}  By \eqref{eq:gamma-lower} and \eqref{eq:gamma-upper},
$\gamma'\geq\omega\gamma^{1/(1-\beta)}$ almost everywhere on $[0,T^{*})$, with
$\omega=\alpha_{1}/\alpha_{2}>0$ depending only on $p$, $q$ and $|\Omega|$.
Since $\gamma\geq\gamma(0)>0$ is absolutely continuous on compact subintervals,
so is $\gamma^{-\beta/(1-\beta)}$, and
\[
\frac{d}{dt}\gamma^{-\frac{\beta}{1-\beta}}
=-\frac{\beta}{1-\beta}\,\gamma^{-\frac{1}{1-\beta}}\gamma'
\leq-\frac{\beta\omega}{1-\beta},
\qquad\text{hence}\qquad
0<\gamma(t)^{-\frac{\beta}{1-\beta}}
\leq\gamma(0)^{-\frac{\beta}{1-\beta}}-\frac{\beta\omega}{1-\beta}\,t
\]
for $0\leq t<T^{*}$, where the left-hand side is positive because
$\gamma\leq H^{1-\beta}+\theta\|u\|_{2}\|u_{t}\|_{2}$ is finite as long as
the solution exists.  Therefore
$T^{*}\leq\frac{1-\beta}{\beta\omega}\,\gamma(0)^{-\beta/(1-\beta)}<\infty$, and
with $\gamma(0)\geq\frac{1}{2}H(0)^{1-\beta}$ and
$H(0)\geq\frac{\varrho^{p+1}}{2(p+1)}\|f\|_{p+1}^{p+1}$ from \eqref{eq:rho0},
\[
T^{*}(\varrho)
\leq\frac{1-\beta}{\beta\omega}\,2^{\frac{\beta}{1-\beta}}H(0)^{-\beta}
\leq\frac{1-\beta}{\beta\omega}\,2^{\frac{\beta}{1-\beta}}
\Bigl(\frac{\|f\|_{p+1}^{p+1}}{2(p+1)}\Bigr)^{-\beta}\varrho^{-(p+1)\beta},
\]
which is \eqref{eq:upper-app} since $(p+1)\beta=\vartheta(p,q)$.
\end{proof}

\begin{remark}\label{rem:BHKstatement}
The upper bounds in \eqref{eq:upper} coincide with those obtained in
\cite{BHK} by the concavity method.  We have reproduced the
argument above in order to track explicitly the dependence of all auxiliary
parameters and constants on the amplitude $\varrho$.  In particular, the
choices \eqref{eq:Lambda-choice}--\eqref{eq:rho0} yield constants independent
of $\varrho$ and the lower estimate
$\gamma(0)\geq\frac{1}{2}H(0)^{1-\beta}$ used in the final integration.  The
lower lifespan estimate obtained in \cite{BHK} is of order $\varrho^{1-p}$;
Theorem \ref{thm:lower} improves this estimate and is proved independently of
\cite{BHK}.
\end{remark}

\section*{Declarations}

\noindent\textbf{Funding.} The author received no specific funding for this
work.

\medskip
\noindent\textbf{Competing interests.} The author declares no competing
interests.

\medskip
\noindent\textbf{Data availability.} No datasets were generated or analysed
during the present study.


\end{document}